\documentclass[a4paper,12pt]{amsart}
\usepackage{diagrams,amssymb}
\newtheorem*{result}{Theorem}
\newtheorem{lemma}{Lemma}
\newtheorem{prop}{Proposition}
\newtheorem{obs}{Observation}
\newtheorem{thm}{Theorem}
\newtheorem{cor}{Corollary}
\theoremstyle{definition}
\newtheorem{defn}{Definition}
\newtheorem{cond}{Condition}

\newtheorem{strgm}{Stratagem}
\theoremstyle{remark}
\newtheorem{rem}{Remark}

\newtheorem{ex}{Example}
\newtheorem{exs}[ex]{Examples}
\newcounter{numl}
\newcommand{\labelnuml}{\textup{(\roman{numl})}}
\newenvironment{numlist}{\begin{list}{\labelnuml}%
{\usecounter{numl}\setlength{\leftmargin}{0pt}%
\setlength{\itemindent}{2\parindent}%
\setlength{\itemsep}{\smallskipamount}\def
\makelabel ##1{\hss \llap {\upshape ##1}}}}{\end{list}}
\newenvironment{bulletlist}{\begin{list}{\labelitemi}%
{\setlength{\leftmargin}{\parindent}\def
\makelabel ##1{\hss \llap {\upshape ##1}}}}{\end{list}}

\DeclareSymbolFont{script}{U}{eus}{m}{n}
\DeclareSymbolFontAlphabet{\mathscr}{script}
\DeclareMathSymbol{\Wedge}{0}{script}{"5E}
\DeclareMathAlphabet{\mathrmsl}{OT1}{cmr}{m}{sl}
\newcommand{\R}{{\mathbb R}}
\newcommand{\C}{{\mathbb C}}
\newcommand{\Z}{{\mathbb Z}}
\newcommand{\N}{{\mathbb N}}

\newcommand{\Sph}{{\mathbb S}}
\newcommand{\T}{{\mathbb T}}

\newcommand{\cC}{{\mathcal C}}

\newcommand{\cE}{{\mathcal E}}
\newcommand{\cF}{{\mathcal F}}

\newcommand{\cI}{{\mathcal I}}

\newcommand{\cL}{{\mathcal L}}

\newcommand{\cN}{{\mathcal N}}

\newcommand{\cS}{{\mathscr S}}
\newcommand{\cU}{{\mathcal U}}
\newcommand{\cY}{{\mathcal Y}}
\newcommand{\g}{\mathfrak g}
\newcommand{\h}{\mathfrak h}

\newcommand{\tor}{{\mathfrak k}}
\newcommand{\torh}{{\mathfrak h}}

\newcommand{\ham}{\mathfrak{ham}}
\newcommand{\con}{\mathfrak{con}}
\newcommand{\Con}{\mathrm{Con}}
\newcommand{\Ab}[1][\cS]{\T_{#1}}
\newcommand{\ab}[1][\cS]{{\mathfrak t}_{#1}}
\newcommand{\bK}{{\mathbf K}}
\newcommand{\Rb}[1]{{\mathscr K}^{#1}}
\newcommand{\Ds}{{\mathscr D}}

\newcommand{\Lv}{L}
\newcommand{\restr}[1]{|_{#1}^{\vphantom x}}

\newcommand{\ip}[1]{\langle #1 \rangle}
\newcommand{\mult}{^{\scriptstyle\times}}
\newcommand{\transp}{^\top}

\renewcommand{\d}{\mathrmsl{d}}
\newcommand{\into}{\hookrightarrow}
\newcommand{\sub}{\subseteq}
\newcommand{\dsum}{\oplus}
\renewcommand{\emptyset}{\varnothing}
\newcommand{\st}{\mathrel{|}}
\newcommand{\eps}{\varepsilon}
\newcommand{\Lam}{{\mathrmsl\Lambda}}

\newcommand{\im}{\mathop{\mathrm{im}}\nolimits}

\newcommand{\rank}{\mathop{\mathrm{rank}}\nolimits}

\newcommand{\grad}{\mathop{\mathrm{grad}}\nolimits}

\newcommand{\Hom}{\mathrm{Hom}}
\newcommand{\Stab}{\mathrm{Stab}}
\newcommand{\stab}{\mathfrak{stab}}
\newcommand{\spns}{\mathrm{span}}

\newcommand{\Proj}{\mathrm P}
\newcommand{\Id}{\mathit{Id}}
\newcommand{\gM}{\cE}
\newcommand{\mm}{{\boldsymbol\sigma}}

\newcommand{\al}{\alpha}

\newcommand{\lamc}{\lambda}
\newcommand{\wt}{\beta}
\newcommand{\afs}{\eps}
\newcommand{\As}{{\mathscr A}}
\newcommand{\Cb}{\Phi}
\newcommand{\Pol}{\Delta}
\newcommand{\Fa}{F}
\newcommand{\e}{e}
\newcommand{\ul}{{\mathbf u}}
\newcommand{\Ll}{{\mathbf L}}

\newcommand{\cG}{\mathcal{G}}
\newcommand{\Gr}{\mathrm{Gr}}
\begin{document}

\title{Toric contact geometry in arbitrary codimension}
\author[V. Apostolov]{Vestislav Apostolov}
\address{Vestislav Apostolov \\ D{\'e}partement de Math{\'e}matiques\\
UQAM\\ C.P. 8888 \\ Succursale Centre-ville \\ Montr{\'e}al (Qu{\'e}bec) \\
H3C 3P8 \\ Canada}
\email{apostolov.vestislav@uqam.ca}
\author[D.M.J. Calderbank]{David M. J. Calderbank}
\address{David M. J. Calderbank \\ Department of Mathematical Sciences\\
University of Bath\\ Bath BA2 7AY\\ UK}
\email{D.M.J.Calderbank@bath.ac.uk}
\author[P. Gauduchon]{Paul Gauduchon}
\address{Paul Gauduchon \\ Centre de Math\'ematiques\\
{\'E}cole Polytechnique \\ UMR 7640 du CNRS\\ 91128 Palaiseau \\ France}
\email{pg@math.polytechnique.fr}
\author[E. Legendre]{Eveline Legendre}
\address{Eveline Legendre\\ Universit\'e Paul Sabatier\\
Institut de Math\'ematiques de Toulouse\\ 118 route de Narbonne\\
31062 Toulouse\\ France}
\email{eveline.legendre@math.univ-toulouse.fr}
\thanks{V.A. is supported in part by IMI/BAS and an NSERC discovery grant. E.L. is partially supported by France ANR project EMARKS No ANR-14-CE25-0010. The authors are grateful to the Institute of Mathematics and Informatics of the Bulgarian Academy of Sciences, the London Mathematical Society, and the Labex CIMI (Toulouse) for hospitality and financial support. They would also like to thank  D. Dragnev for pointing out to them the references \cite{Bolle,Dragnev}, and R. Ponge for drawing their attention to fat distributions in subriemannian geometry.}

\date{\today}
\begin{abstract} We define toric contact manifolds in arbitrary codimension
and give a description of such manifolds in terms of a kind of labelled
polytope embedded into a grassmannian, analogous to the Delzant polytope of a
toric symplectic manifold.
\end{abstract}

\maketitle

\section*{Introduction}

Toric contact geometry lies at the interface of of two important themes in
contemporary geometry. The first is the use of toric methods, beginning with
toric varieties and toric symplectic manifolds, to provide large classes of
geometric objects which can be understood in an essentially combinatorial
way. The second is a growing interest in geometries where the tangent bundle
is not completely reducible, but is instead equipped with a distribution
(subbundle) or filtration. The simplest nontrivial example is a
$(2m+1)$-manifold $N$ equipped with a corank one distribution $\Ds\leq TN$
which is locally the kernel of a \emph{contact form}, i.e., a $1$-form
$\theta$ such that $\theta\wedge\d\theta^{\wedge m}$ is a nonvanishing volume
form. The restriction of $\d\theta$ to $\Ds$ defines a symplectic form on
$\Ds$, and thus $(N,\Ds)$ may be viewed as an odd-dimensional analogue of a
symplectic manifold called a \emph{contact manifold}.

Any contact manifold has a canonical \emph{symplectization}: the total space
of the annihilator $\Ds^0\leq T^*M$ of $\Ds$ inherits an exact $2$-form
$\Omega^\Ds$ from the tautological symplectic structure on $T^*M$, and
$\Omega^\Ds$ is nondegenerate, hence symplectic, on the complement of the zero
section. It follows that contact manifolds, like symplectic manifolds, have no
local invariants, and it is natural to extend symplectic techniques to contact
manifolds. In particular, a contact $(2m+1)$-manifold $N$ is \emph{toric} if
its symplectization is, which means that $N$ has an effective action of an
$(m+1)$-torus preserving the contact distribution $\Ds$, and such toric
contact manifolds have been extensively studied.

Our interest here is in extending this theory to higher codimension, more
precisely, to distributions $\Ds$ of corank $\ell>1$ in $TN$.  Such
distributions arise naturally in CR geometry~\cite{DragTom}, for example on
real submanifolds of codimension $\ell>1$ in a complex manifold, and in
subriemannian geometry~\cite{Mont}.  However, it is not immediately clear how
best to generalize the nondegeneracy condition, as the obstruction to
integrability of $\Ds$ is now a $2$-form on $\Ds$ with values in $TN/\Ds$,
called the \emph{Levi form}. We suggest that the most natural nondegeneracy
requirement for a contact manifold $(N,\Ds)$ in codimension $\ell$ is, as in
codimension one, the local existence of a contact form, which is now a $1$-form
$\theta$ with $\Ds\leq\ker\theta$ and $\d\theta|_\Ds$ nondegenerate. This means
equivalently that the Levi form of $\Ds$ has a nondegenerate component at every
point, which is not as weak as it may seem, because nondegeneracy is a
(fibrewise Zariski) open condition.  Thus any contact manifold has a
symplectization $U_\Ds\sub \Ds^0$ (with symplectic form induced from $T^*N$)
which meets each fibre of $\Ds^0\to N$ in a nonempty Zariski open subset. The
complement of $U_\Ds$ in $\Ds^0$ is then, fibrewise, the cone over a
projective variety (algebraically, a projective hypersurface of degree $m$,
where $\Ds$ has rank $2m$) in $\Proj(\Ds)$, which we call the
\emph{degeneracy variety} of $(N,\Ds)$. In particular, for $\ell>1$, a
contact manifold $(N,\Ds)$ has local invariants except in some low rank cases,
and local classification would involve understanding the deformation theory of
arbitrary projective hypersurfaces.

Despite these difficulties, \emph{toric} contact geometry in higher
codimension turns out to be surprisingly tractable. A contact manifold
$(N,\Ds)$ of dimension $2m+\ell$ and codimension $\ell$ has a symplectization
$U_\Ds$ of dimension $2(m+\ell)$, hence is toric if it has an effective action
of an $(m+\ell)$-torus preserving the contact distribution $\Ds$. In order to
avoid technical difficulties appearing already in codimension $1$, we also
require that there is an $\ell$-subtorus whose orbits are transverse to the
contact distribution. In this setting, we find that despite the presence of
local invariants, toric contact geometry is essentially encoded in a natural
generalization of the momentum polytope of a toric symplectic manifold, namely
a kind of polytope in the grassmannian of $\ell$-dimensional subspaces of the
dual of the Lie algebra of the torus (which is just the projectivization when
$\ell=1$). Thus the theory still has a strong combinatorial flavour, and the
local invariants arise simply because polytopes in a grassmannian are more
flexible than polytopes in an affine or projective space.

In more detail, the contents and main results of the paper are as follows.
After defining contact manifolds and their symplectizations in
\S\ref{s:contact-symplectic} (and giving some simple examples, including
products of odd dimensional spheres), we begin the study of contact actions in
\S\ref{s:lca-tr}. Here we introduce as Condition~\ref{cond:trans}, the
transversality property that we use to avoid some technicalities in the
theory. We also note that a local contact action satisfying this condition
induces a family of $2$-forms on $\Ds$ which we call \emph{Levi
  structures}. In a companion paper~\cite{lrtg}, we use nondegenerate Levi
structures, in the CR context, to construct K\"ahler metrics on quotients by
transverse actions.  While we do not pursue this angle in the present paper,
our approach is in part motivated by it.

We begin our study of contact actions of a torus by introducing the local
theory in \S\ref{s:ab-lca}. Here Levi structures are induced by \emph{Levi
  pairs} $(\g,\lamc)$, where $\g$ is a subalgebra of the infinitesimal torus
$\ab[N]$ acting on the contact manifold $N$ and $\lamc\in\g^*$ picks out a
component of the Levi form.  We introduce fundamental methodology,
Stratagem~\ref{str:setup}, for parametrizing such pairs $(\g,\lamc)$ by maps
to the quotients $\tor=\ab[N]/\g$ and $\torh=\ab[N]/\ker\lamc$. This stratagem
is closely linked to the affine geometry of momentum maps for hamiltonian
torus actions, which we review in Appendix~\ref{a:nmm}. We thus see that a
Levi pair provides an affine slice of the momentum map for the hamiltonian
$\ab[N]$ action on the symplectization $U_\Ds$ of $N$.

The central idea of the paper, the \emph{grassmannian momentum map} of
Definition~\ref{d:gmm}, is a way to package these slices together in a single
compact object (analogous to the polytope of a compact toric symplectic
manifold or orbifold) in the grassmannian $\Gr_\ell(\ab[N]^*)$ of
$\ell$-dimensional subspaces of $\ab[N]^*$, which we call the
\emph{grassmannian image} of $N$.

When the $\ab[N]$ action integrates to an effective contact action of a torus
$\Ab[N]$ on $N$, properties of the grassmannian image can be derived from
fairly standard properties of the $\Ab[N]$ action on $N$ and $U_\Ds$. In
\S\ref{s:ss-torus}, we apply the standard theory of symplectic slices to
obtain a local model for these actions. Then in \S\ref{s:orbit-conv}, we
obtain a connectedness and convexity result for the slices of $\ab[N]$ action
on $U_\Ds$ given by a Levi pair. This is essentially a special case of a
convexity result for transverse symplectic foliations obtained independently
by Ishida~\cite{Ish}, and follows easily from the methods of
Atiyah~\cite{Atiyah}.

These preparations lead us to our main results on toric contact manifolds.  In
section~\S\ref{s:mfd-w-c}, we show that the quotient of $N$ by $\Ab[N]$ is a
simple polytope in the sense that it is at least a manifold with corners.  We
improve this in \S\ref{s:tsc}, where we show that any Levi pair $(\g,\lamc)$
provides a diffeomorphism of the quotient $N/\Ab[N]$ with a simple polytopes
in the affine subspace of $\ab[N]^*$ defined by $(\g,\lamc)$. This shows that
the grassmannian momentum map induces an embedding of $N/\Ab[N]$ into
$\Gr_\ell(\ab[N]^*)$ whose image is polyhedral submanifold with corners as in
Definition~\ref{d:polyhedral}.

We would like to prove that toric contact manifolds are classified by their
grassmannian image, with a suitable labelling of the codimension one faces.
This is true locally by the local models obtained in \S\ref{s:ss-torus} using
symplectic slices.  However, as in toric symplectic geometry~\cite{LT} and
toric contact geometry in codimension one~\cite{Lerman-toric}, there remains a
local-to-global question governed by the cohomology of a sheaf. Thus we have
the following result.

\begin{result} The grassmannian image $\Xi$ of a compact toric contact
manifold of Reeb type $(N,\Ds,\bK)$ is Delzant and polyhedral of Reeb type in
$\Gr_\ell(\ab[N]^*)$, and there is a sheaf $\con^\T(\Ds)$ on $\Xi$ such that
toric contact manifolds of Reeb type with grassmannian image $\Xi$ are
parametrized up to isomorphism by $H^1(\Xi,\con^\T(\Ds))$.
\end{result}

The terminology in this theorem is defined in~\S\ref{s:gr-im}, where we prove
that the grassmannian image has these properties, and establish the
uniqueness, modulo $H^1(\Xi,\con^\T(\Ds))$, of $(N,\Ds,\bK)$ given $\Xi$.
Then in \S\ref{s:exists}, we prove that any such $\Xi$ arises in this way.

In contrast to the symplectic and codimension one cases, $\con^\T(\Ds)$ is the
sheaf of solutions of a linear partial differential equation (for
infinitesimal contactomorphisms), which is overdetermined and typically not
involutive. We show that $H^1(\Xi,\con^\T(\Ds))$ vanishes in special cases,
such as when $(N,\Ds,\bK)$ is a product of codimension one toric contact
manifolds. However, it remains an open question whether it vanishes in
general.

\section{Contact geometry in arbitrary codimension}

\subsection{Levi nondegenerate distributions and symplectization}
\label{s:contact-symplectic}

Let $N$ be smooth manifold of real dimension $2m+\ell$ equipped with a rank
$2m$ vector distribution $\Ds\leq TN$.  Let $\Ds^0$ be the annihilator of
$\Ds$, which is a rank $\ell$ subbundle of $T^*N$. We thus have dual short
exact sequences of vector bundles
\begin{align*}
0\to \Ds &\to TN \xrightarrow{q_D} TN/\Ds\to 0\\
0\to \Ds^0 &\to T^*N \to \Ds^* \to 0,
\end{align*}
where the transpose of the quotient $q_\Ds\colon TN\to TN/\Ds$ identifies
$\Ds^0$ canonically with $(TN/\Ds)^*$: thus we identify an element or section
$\al$ of $(TN/\Ds)^*$ with $\al\circ q_\Ds$ in $\Ds^0$.

\begin{defn} The \emph{Levi form} $\Lv_\Ds\colon \Wedge^2 \Ds\to TN/\Ds$
of $\Ds\leq TN$ is defined, using sections $X,Y\in \Gamma(\Ds)$, by the
tensorial expression
\begin{equation}\label{eq:levi-form}
\Lv_\Ds(X,Y) = - q_\Ds([X,Y]).
\end{equation}
The \emph{nondegeneracy locus} of $\Ds$ is the open subset
\[
U_\Ds=\{\al\in \Ds^0\cong (TN/\Ds)^*\st\al\circ\Lv_\Ds\text{ is nondegenerate}\}
\]
of $\Ds^0$. We say $\Ds$ is \emph{Levi nondegenerate}, and $(N,\Ds)$ is
\emph{contact of rank $m$ and codimension $\ell$}, if $U_\Ds$ has nonempty
intersection with each fibre of $\Ds^0$ over $N$.  A (local) section of
$U_\Ds$ is called a (local) \emph{contact form} on $N$. We often describe
$U_\Ds$ by its complement, which is (the bundle of cones over) the
\emph{degeneracy variety}
\[
V_\Ds=\{[\al]\in\Proj(\Ds^0)\st \al\circ \Lv_\Ds\text{ is degenerate}\},
\]
viewed as a bundle of submanifolds of the fibrewise projectivization
$\Proj(\Ds^0)$ of $\Ds^0$.
\end{defn}

Our normalization of the Levi form is chosen so that for any section $\al$ of
$\Ds^0$, the restriction of $\d\al$ to $\Wedge^2\Ds\leq \Wedge^2TN$ is
$\al\circ \Lv_\Ds$. A Levi nondegenerate distribution $\Ds$ has local frames
$\al_1,\ldots \al_\ell$ for $\Ds^0$ in $U_\Ds$, and hence each fibre of the
degeneracy variety $V_\Ds$ is the projective hypersurface of degree $m$ where
$(\sum_{i=1}^\ell t_i\al_i)\circ\Lv_\Ds$ degenerates (i.e., the determinant, a
homogeneous degree $m$ polynomial in the $\ell$ variables $t_1,\ldots t_\ell$,
vanishes).

\begin{exs}\label{ex:contact} \begin{numlist}
\item The projectivization $\Proj(W)=W\mult/\R\mult\cong\R P^{2m+1}$ of a
  symplectic vector space $(W,\omega)$ (with $\dim W=2m+2$) is contact of
  codimension one, where $\Ds_\xi=\Hom(\xi,\xi^\perp/\xi)\leq
  \Hom(\xi,W/\xi)=T_\xi \Proj(W)$ and $\xi^\perp$ is the orthogonal space to
  $\xi$ with respect to $\omega$; thus $U_\Ds=\Ds^0\setminus 0$ and the
  degeneracy variety $V_\Ds$ is empty. The odd dimensional sphere
  $\Sph^{2m+1}\cong W\mult/\R^+$ (the space of rays in $W$) is also contact,
  being a double cover of $\Proj(W)$.
\item If $(N_i,\Ds_i)$ are contact manifolds, with codimensions $\ell_i$,
for $i\in\{1,\ldots n\}$, then so is $(\prod_{i=1}^n N_i,
\Ds_1\dsum\cdots\dsum \Ds_n)$, with codimension $\ell=\ell_1+\cdots +\ell_n$
and $U_\Ds=\prod_{i=1}^n U_{\Ds_i}$. In particular, the product of $n=\ell$
codimension one spheres $\Sph^{2m_1+1}\times \cdots \times \Sph^{2m_\ell+1}$
is a contact manifold with codimension $\ell$, where each fibre of $U_\Ds$ is
the disjoint union of the $2^\ell$ open $\ell$-quadrants (simple cones) spanned
by $\pm\eta_1,\dots, \pm\eta_\ell$. Each fibre of the degeneracy variety
$V_\Ds$ is thus an union of $\ell$-hyperplanes (i.e., the facets of an
$(\ell-1)$-simplex) with multiplicities $m_i$.
\item Another special case, studied in \cite{Bolle, Dragnev}, is that of {\it
  $\ell$-contact manifolds}. In the terminology of this paper, these are
  contact manifolds $(N,\Ds)$ of rank $m$ and codimension $\ell$ for which the
  Levi form $\Lv_{\Ds}\colon\Wedge^2\Ds\to TN/\Ds$ has rank one (but this
  component is nondegenerate) at every point. Equivalently, each fibre of the
  degeneracy variety $V_\Ds$ is a single hyperplane of multiplicity $m$.
\item The case that the degeneracy variety $V_\Ds$ is empty (i.e., in each
  fibre, it has no real points) has been studied in subriemannian geometry
  as a natural generalization of the codimension one case~\cite{Mont}. This
  condition means equivalently that for any $z\in N$ and any nonzero $X\in
  T_z N$, $\Lv_\Ds(X,\cdot)\colon \Ds_z\to T_zN/\Ds_z$ is surjective. A
  distribution with this property is called~\emph{fat}.
\end{numlist}
\end{exs}

\begin{defn} A diffeomorphism $\Psi$ between contact manifolds $(N,\Ds)$
and $(N',\Ds')$ is called a \emph{contactomorphism} if $\Psi_*(\Ds)
=\Psi^*\Ds'\leq \Psi^* TN'$. We denote by $\con(N,\Ds)$ the Lie algebra of
\emph{infinitesimal contactomorphisms} of $(N,\Ds)$, i.e., the space of vector
fields $X$ on $N$ such that
\begin{equation*}
\cL_X \Gamma(\Ds)\sub\Gamma(\Ds).
\end{equation*}
\end{defn}

We now show that $U_\Ds\sub\Ds^0\leq T^*N$ (or any connected component of
$U_\Ds$) provides a ``symplectization'' of $N$, to which contactomorphisms
lift as hamiltonian vector fields. For this recall that the \emph{tautological
  $1$-form} $\tau$ on $T^*N$ is defined by $\tau_\alpha=\alpha\circ p_*$,
where $\alpha\in T^*N$ and $p_*$ is the differential of the projection
$p\colon T^*N \to N$. The latter fits into an exact sequence
\begin{equation}\label{eq:cot-es}
0\to p^*T^*N\to T(T^*N)\stackrel{p_*}{\to} p^*TN\to 0,
\end{equation}
where (the image of) $p^*T^*N$ is the vertical bundle of $T^*N$. The
\emph{canonical symplectic form} on $T^*N$ is then $\Omega=\d\tau$. For any
local section $\alpha$ of $T^*N$, $\alpha^*\tau=\alpha$ and hence
$\alpha^*\Omega=\alpha^*\d\tau=\d(\alpha^*\tau)=\d\alpha$. Since $\tau$
vanishes on $p^*T^*N$, so does $\Omega$, and the induced pairing between
$p^*T^*N$ and $p^*TN$ the natural one: for any $1$-form $\alpha$ and any lift
$\tilde X$ of a vector field $X$ on $N$, $\Omega(p^*\alpha,\tilde
X)=\d(\tau(\tilde X))(p^*\alpha)=\alpha(X)$, since $\tau(\tilde X)\colon
T^*N\to \R$ is evaluation at $X$, which is linear on the fibres of $p$.

Any vector field $X$ on $N$ admits a \emph{natural lift} to $T^*N$, whose
local flow is the natural lift of the local flow of $X$. For further use, we
recall the following well-known facts.

\begin{prop} The natural lift of any vector field $X$ on $N$ to $TN$ is the
  unique lift $\tilde X$ with $\cL_{\tilde X}\tau=0$; it is hamiltonian with
  respect to $\Omega$, with momentum map $\alpha\mapsto\alpha(X)$, and is
  given by $\tilde X_\alpha=\alpha_*(X)-p^*\cL_X\alpha$ for any extension of
  $\alpha\in T^*N$ to a local section.
\end{prop}
\begin{proof} The natural lift of any diffeomorphism of $N$ to $T^*N$
preserves $\tau$ (by its tautological construction), and hence the natural
lift of $X$ preserves $\tau$. On the other hand for any vector field $\tilde
X$ on $T^*N$ preserving $\tau$, $0=\cL_{\tilde X}\tau= \iota_{\tilde X} \Omega
+ \d(\tau(\tilde X))$, so $\tilde X$ has hamiltonian $\tau(\tilde X)$. If
$\tilde X$ is a lift of $X$ then $\tau(\tilde X)(\alpha)=\alpha(X)$, which
determines $\tilde X$ uniquely from $X$. For the final formula, observe that
$\alpha^*(\iota_{\tilde X}\Omega) =- \d(\alpha(X))
=\iota_X\d\alpha-\cL_X\alpha=
\alpha^*(\iota_{\alpha_*(X)}\Omega-\iota_{p^*\cL_X\alpha}\Omega)$, since
$\alpha^*(\iota_{\alpha_*(X)}\Omega)=\iota_X(\alpha^*\Omega)= \iota _X \d
\alpha$; from this the result follows by nondegeneracy of $\Omega$.
\end{proof}

Let $\tau^\Ds$ be the pullback of $\tau$ to $\Ds^0\leq T^*N$, so that
$\Omega^\Ds:=\d\tau^\Ds$ is the pullback of $\Omega$. Henceforth, $p$ will
denote the bundle projection of $\Ds^0$ rather than $T^*N$ above, and we now
have an exact sequence
\begin{equation}\label{eq:ds0-es}
0\to p^*\Ds^0\to T\Ds^0\stackrel{p_*}{\to} p^*TN\to 0.
\end{equation}

\begin{prop}\label{p:con} For a contact manifold $(N,\Ds)$, $U_\Ds$ is the open
subset of $\Ds^0$ on which $\Omega^\Ds$ is nondegenerate, hence a symplectic
form.  Also, for any contact form $\alpha\colon N\to U_\Ds$,
$TN=\Ds\oplus\Rb\alpha$, where $\Rb\alpha$ is the distribution orthogonal to
$\Ds$ via $\d\alpha$.  The natural lift of any $X\in\con(N,\Ds)$ to $T^*N$ is
tangent to $\Ds^0$, and its restriction defines an isomorphism between
$\con(N,\Ds)$ and the space of projectable vector fields $\tilde X$ on $\Ds^0$
with $\cL_{\tilde X}\tau^\Ds=0$.
\end{prop}
\begin{proof} The restriction of $\Omega^\Ds_\al$ to $V_\al(\Ds^0)\times T_zN$
descends to the canonical nondegenerate pairing between $V_\al(\Ds^0)\cong
\Ds^0_z$ and $T_zN/\Ds_z$.  Hence $\Omega_\alpha$ is nondegenerate if and only
if $\alpha^*\Omega^\Ds=\d\alpha$ is nondegenerate on $\Ds$ for any extension
of $\alpha$ to a local section.  But $\d\alpha|_\Ds=\al\circ\Lv_\Ds$, which
proves the first part. The second part is immediate from the nondegeneracy of
$\d\alpha$ on $\Ds$.

For the last part, observe that the local flow of $X\in\con(N,\Ds)$ preserves
$\Ds$, hence its natural lift to $T^*N$ preserves $\Ds^0$, i.e., the induced
vector field is tangent to $\Ds^0$. The restriction $\tilde X$ to $\Ds^0$ is
clearly a lift of $X$, hence projectable, and $\cL_{\tilde X}\tau^\Ds=0$ since
$\tau^\Ds$ is the pullback of $\tau$. Conversely if a lift $\tilde X$ of $X$
to $\Ds$ satisfies $\cL_{\tilde X}\tau^\Ds=0$ then for any section $Y$ of
$\Ds$, any lift $\tilde Y$ of $Y$ to $\Ds^0$, and any $\alpha\in\Ds^0$,
$0=(\cL_{\tilde X}\tau^\Ds)_\alpha(\tilde Y)=-\tau^\Ds_\alpha([\tilde X,\tilde
  Y]) =\alpha(p_*[\tilde X,\tilde Y])=\alpha([p_*\tilde X,p_*\tilde Y])=
\alpha(\cL_XY)$. Hence $\cL_XY$ is a section of $\Ds$, i.e.,
$X\in\con(N,\Ds)$.  Furthermore, if $f_X\colon \Ds^0\to \R$ is defined by
$f_X(\alpha)=\alpha(X)$ then $\d f_X +\iota_{\tilde X} \Omega^\Ds=0$, which
determines $\tilde X$ uniquely from $X$ on $U_\Ds$, hence on $\Ds^0$, since
$U_\Ds$ is dense in $\Ds^0$.
\end{proof}
We refer to $\Rb\alpha$ as the \emph{Reeb distribution} of $\alpha$.

If we identify $\Ds^0$ with $(TN/\Ds)^*$, then the function $f_X\colon\Ds^0\to
\R$ with $f_X(\alpha)=\alpha(X)$ is the fibrewise linear form defined by
$q_\Ds(X)\in TN/\Ds$. Since $f_X$ is a hamiltonian for the lift $\tilde X$ on
$U_\Ds$, it follows that $\tilde X$ is in fact uniquely determined by
$q_\Ds(X)$.

\begin{rem} The degeneracy variety is a local contact invariant of a contact
manifold $(N,\Ds)$. For example $\Sph^1\times\Sph^5$ and $\Sph^3\times\Sph^3$
cannot be even locally contactomorphic because the degeneracy variety of the
latter has two points in each fibre, whereas the former has only one (with
multiplicity two). Also the lift $\tilde X$ to $\Ds^0$, for any $X\in
\con(N,\Ds)$ must preserve (the cone over) the degeneracy variety. Thus
although $\con(N,\Ds)$ may be identified with a linear subspace of
$\Gamma(TN/\Ds)$, this linear subspace is typically small. 
\end{rem}

\subsection{Local contact actions and transversality}\label{s:lca-tr}

Effective actions of Lie groups by contactomorphisms on $(N,\Ds)$ may be
described locally as follows.

\begin{defn} A (\emph{local, effective}) \emph{contact action} of a Lie
algebra $\g$ on a contact manifold $(N,\Ds)$ is a Lie algebra monomorphism
$\bK\colon\g\to\con(N,\Ds)$.  For $v\in\g$, we write $K_v$ for the induced
vector field $\bK(v)$, and we define $\kappa^\g\colon N\times\g\to TN$ by
$\kappa^\g(z,v)=K_{v,z}$. Let $\Rb\g\leq TN$ be the image of $\kappa^\g$, i.e.,
${\Rb\g}_z:=\spns\{K_{v,z}\st v\in\g\}$.
\end{defn}
\begin{obs}\label{o:lift-ham-act} Let $\bK\colon\g\to\con(N,\Ds)$ be a
local contact action of $\g$ on $(N,\Ds)$ and define $\mu_\g\colon\Ds^0\to
\g^*$ by $\ip{\mu_\g(\al),v}=\al(K_v)$ for $\al\in\Ds^0$ and $v\in \g$. Then
the lift of $\bK$ to $T^*N$ preserves $\Ds^0$, and the induced local action
$\tilde\bK$ is hamiltonian on $U_\Ds$ with momentum map
$\mu_\g\restr{U_\Ds}$\textup; in particular $\ip{\d\mu_\g(\tilde
  K_v),w}=-\ip{\mu_\g,[v,w]_\g}$ for all $v,w\in\g$.
\end{obs}
This is immediate from the discussion at the end
of~\S\ref{s:contact-symplectic}.

\begin{ex}\label{con-pb} Let $G$ be a Lie group of dimension $\ell$ with Lie
algebra $\g$ and let $\pi\colon N\to M$ be a principal $G$-bundle with
connection $\eta\colon TN\to \g$, where $\dim M=2m$.  Then $\Ds:=\ker\eta$ is
a rank $2m$ distribution on $N$, $G$ acts on $(N,\Ds)$ by contactomorphisms,
and $\eta$ induces a bundle isomorphism of $TN/\Ds$ with $N\times\g$.  In this
trivialization, the Levi form of $\Ds$ is $\d\eta+\frac12[\eta\wedge\eta]_\g$,
the pullback to $N$ of the curvature $F^\eta$ of $\eta$.
\end{ex}
Suppose $\bK\colon\g\to\con(N,\Ds)$ is a local contact action where $(N,\Ds)$
is contact of codimension $\ell=\dim\g$. Abstracting the local geometry of
Example~\ref{con-pb}, we say the action of $\g$ is \emph{transversal} if the
following condition holds.

\begin{cond}\label{cond:trans} At every point of $N$, $\Ds+\Rb\g=TN$.
Equivalently\textup:
\begin{numlist}
\item $\rank\Rb\g=\ell$ everywhere on $N$\textup;
\item $\Ds\cap\Rb\g$ is the zero section of $TN$ (and thus
  $TN=\Ds\oplus\Rb\g$).
\end{numlist}
\end{cond}

The composite $q_\Ds\circ\kappa^\g\colon N\times\g\to\Rb\g\to TN/\Ds$ is a
bundle isomorphism and so there is a canonically defined $1$-form
$\eta^\g\colon TN\to \g$, characterized by
\begin{equation*}
\ker\eta^\g=\Ds\quad \text{and}\quad \forall\, v\in \g, \quad \eta^\g(K_v)=v. 
\end{equation*}

We also denote by $\eta^\g$ the induced map from $TN/\Ds$ to $\g$. We can now
verify the usual properties of a connection $1$-form.

\begin{lemma}\label{l:eta} Let $\eta=\eta^\g$ denote the $\g$-valued $1$-form
of a contact action on $N$ satisfying Condition~\textup{\ref{cond:trans}}. Then
for any $v\in\g$, $\cL_{K_v}\eta+[v,\eta]_\g=0$, and
$\d\eta+\frac12[\eta\wedge\eta]_\g=\eta\circ\Lv_\Ds$, where $\Lv_\Ds$ is
extended by zero from $\Ds$ to $TN=\Ds\dsum\Rb\g$.
\end{lemma}
\begin{proof} For any $v\in\g$, $K_v$ preserves $\Ds$ and $\Rb\g$, hence also
$\kappa^\g\circ\eta\colon TN\to\Rb\g$. Thus
\begin{equation*}
0=\cL_v (\bK\circ \eta)(X)=(\cL_v\bK) (\eta(X))+\bK(\cL_{K_v}\eta)(X).
\end{equation*}
Since $\bK$ is a morphism of Lie algebras, $(\cL_{K_v}\bK)(w)=[K_v,K_w]=
\bK([v,w]_\g)$, so $(\cL_{K_v}\eta)(X)+[v,\eta(X)]_\g=0$ by the injectivity of
$\bK$. Now for $X,Y\in \Gamma(\Ds)$ and $v,w\in \g$, $\d\eta(X+K_v,Y+K_w)=
-\eta([X+K_v,Y+K_w]) = -\eta([X,Y])-[\eta(K_v),\eta(K_w)]_\g =
\eta(\Lv_\Ds(X+K_v,Y+K_w)) - [\eta(X+K_v),\eta(Y+K_w)]_\g$.
\end{proof}

Since $\bK\colon\g\to\con(N,\Ds)$ is a Lie algebra morphism, $\Rb\g$ is an
integrable distribution, equal to $\Rb\alpha$ for any contact form $\alpha$
such that $\alpha(K_v)$ is constant for all $v\in\g$.  For any $\lamc\in\g^*$,
define $\eta^\lamc=\eta^{\g,\lamc}\colon N\to\Ds^0$ by $\eta^\lamc_z(X)
=\ip{\eta_z(X),\lamc}$, so that $\eta^\lamc(K_v)=\lamc(v)$ and
$\d\eta^\lamc\restr{\Ds}=\ip{\d\eta\restr{\Ds},\lamc}=\eta^\lamc\circ\Lv_\Ds$
is the $\lamc$-component of the Levi form of $\Ds$.

\begin{defn} For a contact action $\bK\colon\g\to\con(N,\Ds)$ satisfying
Condition~\ref{cond:trans}, we refer to $\Rb\g$ and its integral submanifolds
as the associated \emph{Reeb distribution} and \emph{Reeb foliation}
transverse to $\Ds$.  For $\lamc\in[\g,\g]^0\leq\g^*$, we call
$(\Ds,\d\eta^\lamc\restr{\Ds})$ the induced \emph{Levi structure}; it is said
to be \emph{nondegenerate} $\d\eta^\lamc\restr{\Ds}$ is, i.e., $\eta^\lamc$ is
a contact form \textup($U_\Ds$-valued\textup).
\end{defn}

\section{Contact torus actions}

\subsection{Abelian local contact actions}\label{s:ab-lca}

Suppose that $\bK\colon\ab[N]\to \con(N,\Ds)$ is a local contact action of an
abelian Lie algebra $\ab[N]$ on a contact manifold $(N,\Ds)$. Let
\begin{align*}
\kappa&:=\kappa^{\ab[N]}\colon N\times\ab[N]\to TN, && \text{with}&
\kappa(z,v)&=K_{v,z},\quad\text{and}\\
\mu&:=\mu_{\ab[N]}\colon\Ds^0\to \ab[N]^*,&& \text{with}&
\ip{\mu(\al),v}&=\al(K_v),
\end{align*}
so that $(p,\mu)\colon\Ds^0\to N\times\ab[N]^*$ is the pointwise transpose of
$q_\Ds\circ\kappa\colon N\times\ab[N]\to TN/\Ds$.  By
Observation~\ref{o:lift-ham-act}, $\bK$ lifts to a hamiltonian action on
$U_\Ds$ with momentum map $\mu\restr{U_\Ds}$. The following is the main tool
in our analysis of abelian local contact actions.
\begin{defn} An $\ell$-dimensional subalgebra $\iota_\g\colon\g\into\ab[N]$ and
an element $\lamc\in\g^*\setminus 0$ together form a \emph{Levi pair}
$(\g,\lamc)$ for $\bK$ if:
\begin{bulletlist}
\item $\g$ acts transversally on $N$ via $\bK$, i.e.,
  $\Rb\g:=\spns\{K_{v,z}\st v\in\g\}$ satisfies Condition~\ref{cond:trans}.
\end{bulletlist}
Let $\eta\colon TN\to \g$ be the connection $1$-form of $\g$; we say
$(\g,\lamc)$ is \emph{nondegenerate} if:
\begin{bulletlist}
\item $\eta^\lamc=\ip{\eta,\lamc}$ is a contact form, i.e.,
  $(\Ds,\d\eta^\lamc\restr{\Ds})$ is a nondegenerate Levi structure.
\end{bulletlist}
Thus $(p,\mu_\g)\colon\Ds^0\to N\times\g^*$, with
$\mu_\g:=\iota_\g\transp\circ\mu$, is an isomorphism.  We say $(N,\Ds,\bK)$ is
\emph{Reeb type} if it admits a nondegenerate Levi pair.
\end{defn}

Let $(\g,\lamc)$ be a Levi pair. For any $v\in\ab[N]$,
$(\d\eta^\lamc)(K_v,\cdot) =-\d(\eta^\lamc(K_v))$.  We may thus view
$\eta^\lamc(K_v)= \ip{\mu(\eta^\lamc),v}$ as the ``horizontal momentum'' of
$K_v$ with respect to the Levi structure $(\Ds,\d\eta^\lamc\restr{\Ds})$.
Observe that if $v\in\g$, $\eta^\lamc_z(K_v)= \ip{v,\lamc}$, which vanishes
for $v\in\ker\lamc\leq\g$.  Hence $z\mapsto \mu(\eta^\lamc_z)\in \ab[N]^*$
takes values in $(\ker\lamc)^0\cong (\ab[N]/\ker\lamc)^*$.

\begin{strgm}\label{str:setup} For any pair $(\g,\lamc)$ with $\g\leq\ab[N]$
and $\lamc\in\g^*\setminus 0$, the quotient $\ab[N]/\ker\lamc$ is an extension
by $\R$ of the quotient $\ab[N]/\g$.  To allow $(\g,\lamc)$ to vary, it is
convenient to fix this extension $\torh\to\tor$ (where $\torh$ and $\tor$ are
abelian Lie algebras of dimensions $n-\ell+1$ and $n-\ell$ for $n=\dim
\ab[N]$); then the commutative diagram
\begin{diagram}[size=1.5em,nohug]
0 &\rTo & \g &\rTo^{\iota_\g} & \ab[N] & \rTo^\ul & \tor &\rTo &0\\
  &     & \dTo^\lamc&   & \dTo^\Ll &   & \dEq & \\
0 &\rTo & \R&\rTo^{\afs}&\torh & \rTo^\d & \tor &\rTo &0.
\end{diagram}
of short exact sequences associates pairs $(\g,\lamc)$, with $\ab[N]/
\ker\lamc\cong\torh$, to surjective linear maps $\Ll\colon\ab[N]\to \torh$
(thus $\g$ is the kernel of $\ul:=\d\circ\Ll$, and $\lamc$ is induced by
$\Ll\restr\g$).

Let $\As\sub\torh^*$ be the affine subspace $(\afs\transp)^{-1}(1)$ modelled
on $\tor^*$; then $\torh$ may be identified with the affine linear functions
$\ell\colon\As\to\R$, whence $\d\ell\in \tor$ is the linear part of
$\ell\in\torh$.
\end{strgm}

If $(\g,\lamc)$, defined by $\Ll\colon \ab[N]\to\torh$, is a Levi pair, then
the map $\mu^\lamc\colon N\to\As\sub\torh^*$, determined uniquely by the
formula
\begin{equation}\label{eq:momentum}
\ip{\mu^\lamc(z),\Ll(v)}= \eta^\lamc_z(K_v)
\end{equation}
for all $z\in N$ and $v\in \ab[N]$, will be called the \emph{horizontal
  \textup(natural\textup) momentum map} of $(\Ds,\d\eta^\lamc\restr{\Ds})$
(cf.~Appendix~\ref{a:nmm}). Equivalently the diagram
\begin{diagram}[size=1.5em,nohug]
N & \rTo^{\eta^\lamc}    & U_{\Ds} \\
\dTo^{\mu^\lamc}&        & \dTo_\mu \\
\torh^*&\rTo^{\Ll\transp}& \ab[N]^*
\end{diagram}
commutes, i.e., $\Ll\transp\circ\mu^\lamc= \mu\circ \eta^\lamc$.

\begin{lemma}\label{l:lrt} Let $(N,\Ds,\bK)$ be a contact manifold with an
abelian local contact action. Then the following are equivalent.
\begin{numlist}
\item $(N,\Ds,\bK)$ admits a nondegenerate Levi pair in a neighbourhood of
any point.
\item $(p,\mu)\colon \Ds^0\to N\times\ab[N]^*$ is injective.
\item $q_\Ds\circ\kappa\colon N\times\ab[N]\to TN/\Ds$ is surjective.
\end{numlist}
\end{lemma}
\begin{proof} The last two conditions are equivalent because $(p,\mu)$ is the
transpose of $q_\Ds\circ\kappa$. Over any open neighbourhood where there is a
nondegenerate Levi pair, $(p,\mu_\g)$ is an isomorphism and hence $(p,\mu)$ is
injective. Conversely, if $q_\Ds\circ\kappa$ surjects, a complement to the
kernel at a point gives a locally transversal subalgebra
$\iota_\g\colon\g\into\ab[N]$, hence also a Levi pair: $U_\Ds$ is open with
nonempty fibres, so we can find $\lamc\in\g^*$ such that $\eta^{\lamc}$ is
locally a contact form.
\end{proof}
\begin{rem}\label{r:open} If $N$ is compact, $\{\lamc\in\g^*\setminus 0
\st \eta^\lamc\in\Gamma(U_\Ds)\}$ is an open cone $\cC_\g\sub \g^*$.
\end{rem}
\begin{defn}\label{d:lrt} $(N,\Ds,\bK)$ is \emph{locally Reeb type} if the
conditions of Lemma~\ref{l:lrt} hold.
\end{defn}

For a vector space $\ab[]$ of dimension $n\geq\ell$, we denote by
$\Gr_\ell(\ab[]^*)$ the grassmannian of $\ell$-dimensional subspaces of
$\ab[]^*$.  This is a manifold of dimension $\ell(n-\ell)$ with distinguished
\emph{affine charts} defined as follows. Let $\iota_\g\colon \g
\hookrightarrow \ab[]$ be an $\ell$-dimensional subspace of $\ab[]$ and let
$\Gr_\ell^\g(\ab[]^*)=\{\xi\in\Gr_\ell(\ab[]^*)\st \iota\transp_\g(\xi)
=\g^*\}$ be set of subspaces $\xi\leq\ab[]^*$ complementary to the kernel of
$\iota\transp_\g\colon \ab[]^*\to \g^*$. Then $\Gr_\ell^\g(\ab[]^*)$ is an
open subset of $\Gr_\ell(\ab[]^*)$, and is isomorphic to the affine space
$\{\theta\in\Hom(\g^*,\ab[]^*)\st \iota_\g\transp\circ\theta=\Id_{\g^*}\}$ via
the map sending $\theta$ to $\im\theta$. The inverse sends $\xi\in
\Gr_\ell^\g(\ab[]^*)$ to $\psi_\g(\xi)\colon \g^*\to \ab[]^*$ where
\begin{equation}\label{eq:theta}
\text{for }\lamc\in\g^*,
\; \ip{\psi_\g(\xi),\lamc}\in\ab[]^*\text{ is the unique point in }
(\iota_\g\transp)^{-1}(\lamc)\cap\xi.
\end{equation}

\begin{defn}\label{d:gmm}
Let $\Rb{}=\Rb{\ab[N]}=\im\kappa$, $\gM:=\im(p,\mu)\leq N\times\ab[N]^*$, and 
$\Theta:=\gM^0=\ker(q_\Ds\circ\kappa)= \kappa^{-1}(\Ds)\leq N\times\ab[N]$.
If $(N,\Ds,\bK)$ is locally Reeb type, $\Rb{}\cap\Ds$ has codimension $\ell$
in $\Rb{}$ and $\gM$ is a rank $\ell$ subbundle of $N\times\ab[N]^*$ (with
$\gM^*\cong (N\times\ab[N])/\Theta\cong TN/\Ds$). Thus $\gM$ defines a smooth
map $N\to \Gr_\ell(\ab[N]^*); z\mapsto \gM_z$ which we call the
\emph{grassmannian momentum map} of $(N,\Ds,\bK)$, and also denote by $\gM$.
\end{defn}

\begin{rem}\label{r:affine} If $(\g,\lamc)$ is a nondegenerate Levi pair
(over an open subset of $N$) then $\gM$ takes values in the affine chart
  $\Gr_\ell^\g(\ab[N]^*)$ (on that open subset) and $\ip{\psi_\g\circ\gM,\lamc}
=\mu\circ \eta^\lamc = \Ll\transp\circ\mu^\lamc$.
\end{rem}

\begin{ex} Let us revisit Examples~\ref{ex:contact}.
\begin{numlist}
\item The codimension one case is particularly straightforward
as $\Gr_1(\ab[N]^*)$ is just the projectivization of $\ab[N]^*$ and the
grassmannian momentum map $\gM\colon N\to\Proj(\ab[N]^*)$ is just the
projectivization of the momentum map $\mu$ on $U_\Ds$, i.e.,
$\gM_{p(\alpha)}=[\mu(\alpha)]$. In particular, if $N=\Sph^{2m+1} =W\mult/\R^+$
and $\ab[N]$ is the Lie algebra of the diagonal $(m+1)$-torus with respect to
a basis of $W$, then $\im\mu$ is the union of the standard quadrant in
$\ab[N]^*$ and its negation. Hence $\im\gM$ is the projectivization of this
quadrant, the standard simplex.

\item Let $N=\Sph^{2m_1+1}\times \cdots \times \Sph^{2m_\ell+1}$ be a product
of $\ell$ contact spheres $\Sph^{2m_i+1}=W_i\mult/\R^+$, with the product
local contact action of an abelian Lie algebra $\ab[N]=\bigoplus_{i=1}^\ell
\ab[i]$, where $\ab[i]$ has dimension $m_i+1$ and acts diagonally on $W_i$ as
in (i).  Then, via the natural embedding of $\prod_{i=1}^\ell \Proj(\ab[i]^*)$
into $\Gr_\ell(\ab[N]^*)$, sending $(U_1,\ldots U_\ell)$ to $U_1\dsum \cdots
\dsum U_\ell$, the image of $\gM$ is a product of simplices in
$\Gr_\ell(\ab[N]^*)$.
\end{numlist}
\end{ex}

\begin{prop}\label{p:transverse} Let $\bK\colon\ab[N]\to \con(N,\Ds)$ be
an abelian local contact action. If $\bK$ is locally Reeb type, the lift
$\tilde\bK$ of $\bK$ to $\Ds^0\to N$ is transverse to the fibres, i.e., if
$K_{v,z}=0$ \textup(for $v\in\ab[N], z\in N$\textup) then $\tilde K_v$ is
identically zero along $\Ds^0_z=p^{-1}(z)$.
\end{prop}
\begin{proof} Since $(p,\mu)\colon \Ds^0\to N\times\ab[N]^*$ is a monomorphism
of vector bundles, it is an immersion. However, for any $v\in \ab[N]$,
$\d\mu(\tilde K_v)=0$ by equivariance of $\mu$, so if $K_{v,z}=p_*(\tilde
K_v)_z=0$, then $\d(p,\mu)(\tilde K_v)=0$ along $p^{-1}(z)$, hence $\tilde
K_v=0$ along $p^{-1}(z)$.
\end{proof}
Hence for any $\alpha\in \Ds^0$, the infinitesimal stabilizer
$\stab_{\ab[N]}(\alpha)=\stab_{\ab[N]}(p(\alpha))$.  Since $\ker \kappa\leq
\ker q_\Ds\circ \kappa=\Theta$, for any $z\in N$, $\stab_{\ab[N]}(z)\leq
\Theta_z$ and hence $\gM_z\leq\stab_{\ab[N]}(z)^0$.

\subsection{Symplectic slices for torus actions}\label{s:ss-torus}

Let $(N,\Ds)$ be a contact manifold of rank $m$ and codimension $\ell$, and
let $\Ab[N]=\ab[N]/2\pi\Lam$ be a (real) torus with (abelian) Lie algebra
$\ab[N]=\Lam\otimes_\Z\R$, where $\Lam$ is the lattice of circle subgroups of
$\Ab[N]$.

\begin{defn} A \emph{contact torus action} of $\Ab[N]$ on $N$ is a local
contact action $\bK\colon\ab[N]\to \con(N,\Ds)$ which integrates to an
effective (i.e., faithful) action of $\Ab[N]$.
\end{defn}

We recall the construction of a symplectic slice along an $\Ab[N]$-orbit
$\Ab[N]\cdot \alpha$ for $\alpha\in U_\Ds$. The isotropy representation of
$\Stab_{\Ab[N]}(\alpha)$ on $T_\alpha U_\Ds$ induces an effective linear
action, symplectic with respect to $\omega =\Omega^\Ds_\alpha$, of any
subtorus $H\leq \Stab_{\Ab[N]}(\alpha)$. Since $\Ab[N]\cdot\alpha$ is
isotropic with respect to $\Omega^\Ds$, the isotropy representation has an
$H$-invariant filtration
\begin{equation}\label{eq:filt}
0\leq T_\alpha(\Ab[N]\cdot \alpha)\leq T_\alpha(\Ab[N]\cdot \alpha)^{\perp}
\leq T_\alpha U_\Ds,
\end{equation}
where ${}^{\perp}$ denotes the orthogonal space with respect to
$\omega=\Omega^\Ds_\alpha$.

\begin{rem}\label{rem:isotropy-rep} There is also~\cite{Atiyah,Delzant,LT}
an $H$-invariant $\omega$-compatible complex structure on $T_\alpha U_\Ds$ and
a decomposition $(T_\alpha U_\Ds, \omega) = \bigoplus_{i=0}^k (W_i,\omega_i)$
into $H$-invariant $\omega$-orthogonal symplectic subspaces, so the induced
representation of $H$ (or its Lie algebra $\h$) on $W_i$ has weight $\wt_i\in
\h^*$ for $i\in\{0,\ldots k\}$.  We assume $\wt_0=0$ and $\wt_i\neq 0$ for
$i>0$, so $W_0\leq T_\alpha U_\Ds$ is the tangent space to the fixed point set
of $H$ at $\alpha$.  Since the $H$-action is effective, the weights
$\wt_1,\ldots \wt_k$ span the \emph{weight lattice} of $\h^*$, dual to the
lattice $\Lam\cap\h$. Note that $T_\alpha(\Ab[N]\cdot\alpha)\leq W_0$ and
$T_\alpha(\Ab[N]\cdot\alpha)^{\perp}+W_0=T_\alpha U_\Ds$.
\end{rem}

When $H$ is the identity component $H_\alpha$ of $\Stab_{\Ab[N]}(\alpha)$ and
$\h_\alpha:=\stab_{\ab[N]}(\alpha)$,
\begin{equation*}
T_\alpha(\Ab[N]\cdot \alpha)\cong\ab[N]/\h_\alpha \quad\text{and}\quad
T_\alpha U_\Ds/T_\alpha(\Ab[N]\cdot\alpha)^{\perp}\cong
(\ab[N]/\h_\alpha)^*\cong \h_\alpha^0\leq \ab[N]^*.
\end{equation*}
\begin{defn} \label{d:symp-iso-rep} The middle composition factor
in~\eqref{eq:filt},
\begin{equation*}
W^\alpha:=T_\alpha(\Ab[N]\cdot \alpha)^{\perp}/T_\alpha(\Ab[N]\cdot \alpha),
\end{equation*}
is called the \emph{symplectic isotropy representation}.  By
Remark~\ref{rem:isotropy-rep}, $W^\alpha\cong W^\alpha_0\dsum
\bigoplus_{i=1}^k W_i$, where $W^\alpha_0= \bigl(T_\alpha(\Ab[N]\cdot
\alpha)^{\perp}\cap W_0\bigr)/T_\alpha(\Ab[N]\cdot \alpha)$ and $W_i$ are the
nonzero weight spaces of $H_\alpha$ (with weights $\beta_i\in\h_\alpha^*$).
The momentum map $\mu_W\colon W^\alpha\to \h_\alpha^*$ of the
$H_\alpha$-action is
\begin{equation}\label{eq:mu-W}
\mu_W(w_0,w_1,\ldots w_k) =\sum_{i=1}^k |w_i|^2 \wt_i
\end{equation}
for all $w_0\in W^\alpha_0$ and $(w_1,\ldots w_k)\in \bigoplus_{i=1}^k W_i$.
The image of $\mu_W$ is thus the convex cone $\cC^\alpha$ in $\h_\alpha^*$
generated by the nonzero weights $\beta_1,\ldots\beta_k$ of $W^\alpha$.
\end{defn}

To construct a model for a neighbourhood of $\Ab[N]\cdot \alpha$, observe that
$T^*\Ab[N]\times W^\alpha \cong\Ab[N]\times\ab[N]^*\times W^\alpha$ is a
symplectic manifold with commuting hamiltonian actions of $\Ab[N]$ (the right
action on $T^*\Ab[N]$) and $\Stab_{\Ab[N]}(\alpha)$ (the diagonal left
action), where the latter has momentum map $(g,\xi,v)\mapsto
\mu_W(v)-\xi|_{\h_\alpha}$. By the orbit-stabilizer theorem $\Ab[N]$ is a
principal $\Stab_{\Ab[N]}(\alpha)$-bundle over
$\Ab[N]\cdot\alpha\cong\Ab[N]/\Stab_{\Ab[N]}(\alpha)$, and if we choose a
splitting $\chi\colon\ab[N]\to\h_\alpha$, we may identify the symplectic
quotient of $T^*\Ab[N]\times W^\alpha$ by $\Stab_{\Ab[N]}(\alpha)$ at the zero
momentum level ($\xi|_{\h_\alpha}=\mu_W(v)$) with the associated vector bundle
\begin{equation}\label{eq:slice-model}
\cY:=\Ab[N]\times_{\Stab_{\Ab[N]}(\alpha)} (W^\alpha\dsum\h_\alpha^0)
\end{equation}
over $\Ab[N]/\Stab_{\Ab[N]}(\alpha)$. The induced action of $\Ab[N]$ is
hamiltonian with momentum map
\begin{equation}\label{eq:mu-chi}
\mu_\chi([g,v+\xi])= \mu(\alpha)+\chi\transp(\mu_W(v))+\xi,
\end{equation}
where $g\in\Ab[N]$, $v\in W^\alpha$ and $\xi\in\h_\alpha^0$.  The Symplectic
Slice Theorem~\cite{Audin,KL,Lerman-conv,Lerman-toric,LT} now asserts the
following.
\begin{lemma} \label{l:sslice} For any splitting $\chi$, there is
a $\Ab[N]$-invariant neighbourhood $\cU$ of the orbit $\Ab[N]\cdot\alpha\sub
U_\Ds$ and a $\Ab[N]$-equivariant symplectomorphism $\Psi_\chi$ from
$(\cU,\Omega)$ to a neighbourhood of the zero section of $(\cY,\Omega_\chi)$
such that $\mu\restr\cU=\mu_\chi\circ\Psi_\chi$.

Thus $\mu-\mu(\alpha)$ maps $\cU$ to a neighbourhood of
$\,0\in\h_\alpha^0\oplus\chi\transp(\cC^\alpha)$.
\end{lemma}

\subsection{Orbit stratification and convexity}\label{s:orbit-conv} 

The theory of effective proper abelian group actions~\cite{Audin,GGK} implies
that for $H\leq\Ab[N]$,
\begin{equation*}
N_{(H)} = \{ z\in N \st H = \Stab_{\Ab[N]}(z) \}\leq
N^H = \{ z\in N \st H \leq \Stab_{\Ab[N]}(z)\}
\end{equation*}
is an open submanifold of a closed submanifold of $N$, and if $N_{(H)}$ is
nonempty (for which $H$ must be a closed subgroup) then it is dense in the
fixed point set $N^H$ of $H$.

\begin{defn}\label{d:Ncomb} The connected components of $N_{(H)}$,
for $H\leq\Ab[N]$, and their closures (which, if nonempty, are connected
components of $N^H$) are called the \emph{open} and \emph{closed orbit strata}
of $(N,\bK)$.  (Thus the open orbit strata partition $N$.) The
\emph{combinatorics} $\Cb_N$ of $(N,\bK)$ is the poset of closed orbit strata,
partially ordered by inclusion.
\end{defn}

\begin{prop}\label{p:orbit-strata} Let $\bK$ be a contact torus action
of $\Ab[N]$ on a contact manifold $(N,\Ds)$ of codimension $\ell$ with locally
Reeb type, let $i\colon N'\to N$ be the inclusion of a closed orbit stratum
$N'\in\Cb_N$, let $\Ds'=TN'\cap\Ds$, and let $H$ be the kernel of the induced
$\Ab[N]$-action on $N'$. Then $(N',\Ds')$ is a contact manifold of codimension
$\ell$, with $U_{\Ds'}\cong i^* U_\Ds$, and $\bK$ induces a contact torus
action of $\Ab[N]/H$ on $N'$ which is locally Reeb type.
\end{prop}
\begin{proof} The local Reeb type condition means that $\Ds\cap\Rb{}$ has
codimension $\ell$ in $\Rb{}$ (the tangent distribution to the $\Ab[N]$
orbits). Since $i^*\Rb{}\leq TN'$, $\Ds'$ has codimension $\ell$ in $TN'$, and
$i_*\transp\colon i^*T^*N\to T^*N'$ restricts to an isomorphism
$i^*\Ds^0\to(\Ds')^0$, which identifies $U_{\Ds'}$ with $i^* U_\Ds$ (since the
induced map $(\Ds')^0\to\Ds^0$ pulls the tautological $1$-form $\tau^\Ds$ back
to the corresponding $\tau^{\Ds'}$). Hence $(N',\Ds')$ is contact and the rest
is immediate.
\end{proof}
Note that $H=\Stab_{\Ab[N]}(z)$ for any $z$ in the open orbit stratum
corresponding to $N'$.

If a contact torus action $\bK$ of $\Ab[N]$ on $N$ has locally Reeb type,
Proposition~\ref{p:transverse} implies that the hamiltonian $\Ab[N]$-action on
$U_\Ds$ is transverse to the fibres of $p\colon\Ds^0\to N$. Thus for any
subtorus $H\leq\Ab[N]$, the fixed point set of $H$ in $U_\Ds$ is
$p^{-1}(N^H)$.

Recall that the \emph{critical submanifolds} of a smooth function $f\colon
N\to\R$ are the connected components of the zero-set of $\d f$; then $f$ is
called a \emph{Morse--Bott function} if along any critical submanifold, its
transverse hessian is nondegenerate.

\begin{lemma}\label{l:Bott} Suppose $(N,\Ds,\bK)$ has Reeb type. Then for any
nondegenerate Levi pair $(\g,\lamc)$ and any $v\in\ab[N]$,
$f:=\eta^\lamc(K_v)$ is a Morse--Bott function on $N$ whose critical
submanifolds all have even index.
\end{lemma}
\begin{proof}[Proof cf.~\cite{Atiyah}]  Since $\d(\eta(K_v))
=-(\d\eta)(K_v,\cdot)$ and $\d\eta^\lamc$ has kernel $\Rb\g$, a point $z\in N$
is critical for $f$ iff $K_{v,z}\in \Rb\g$, in which case $\d(\eta(K_v))_z=0$.
This holds iff $\exists\,w\in \g$ such that $K_{v-w}$ vanishes at $z$, and
since $\eta^\lamc(K_{v-w})=f-\ip{w,\lamc}$ has the same critical submanifolds
as $f$, we may assume $w=0$. Let $Z$ be the critical submanifold of $f$
containing $z$; since $\eta(K_v)_z=0$ and $\d(\eta(K_v))\restr{Z}=0$, it
follows that $K_v$ vanishes along $Z$ and hence $Z$ is a component of the
fixed point set $N^H$ of the subtorus $H$ of $\Ab[N]$ generated by
$\exp(K_v)$.  By Remark~\ref{rem:isotropy-rep}, with $\alpha=\eta^\lamc_z$,
the tangent space $T_\alpha U_\Ds$ decomposes into symplectic weight spaces
and the zero weight space is the tangent space to $p^{-1}(Z)$.  Hence the
normal bundle to $Z$ in $N$ is isomorphic to the sum of the nonzero weight
spaces. The transverse hessian of $f$ at $z$ is the hermitian form
$(w_1,\ldots w_k)\mapsto\sum_i \wt_i(v) |w_i|^2$, which is nondegenerate of
even index because $K_v$ generates $H$.
\end{proof}

\begin{prop}\label{p:convex} Let $(N,\Ds)$ be a compact connected contact
manifold of rank $m$ and codimension $\ell$ with a contact action of a torus
$\Ab[N]$ with Lie algebra $\ab[N]$. Assume that the action is Reeb type
with Levi pair $(\g,\lamc)$.

For any $v_1,\ldots v_k\in \ab[N]$, the map $f\colon N \to \R^k$ with
components $f_i = \eta^\lamc(K_{v_i})$ satisfies
\begin{itemize}
\item[$(A)$] all fibres $f^{-1}(x)$ are empty or connected\textup;
\item[$(B)$] the image $f(N)$ is convex.
\end{itemize}
Further, if $Z_j:j\in\cI$ are the connected components of the set of common
critical points of $f_i$, then $f(Z_j)=\{p_j\}$ is a single point and $f(N)$
is the convex hull of $\{p_j\st j\in\cI\}$.
\end{prop}
The proof is essentially identical to the proof of the symplectic convexity
theorem by M. Atiyah~\cite{Atiyah} (see also~\cite{GS}): the key ingredient is
Lemma~\ref{l:Bott}, which makes essential use of Condition~\ref{cond:trans}
(the transversality property of the subalgebra $\g$).  Without this
assumption, convexity may fail even in codimension one~\cite{Lerman-cuts},
although in that case, the conditions needed for convexity are well
understood~\cite{ChiKar,Lerman-conv,Lerman-toric,Lerm-geod-flow}. A similar
convexity result for transverse symplectic foliations has been obtained
recently by Ishida~\cite{Ish}.

\section{Toric contact manifolds}

\subsection{The quotient manifold with corners}\label{s:mfd-w-c}

If $(N,\Ds)$ is contact of rank $m$, then $\rank\Rb{}\cap\Ds\leq m$, and hence
$\dim\ab[N]\leq m+\ell$. On any open set where $\rank\Rb{}=m+\ell$,
$(N,\Ds,\bK)$ is locally Reeb type.

\begin{defn} A compact contact manifold $(N,\Ds)$ of rank $m$ and
codimension $\ell$ with a contact torus action $\bK$ of $\Ab[N]$ is
\emph{toric} if it is locally Reeb type and $\dim\Ab[N]=m+\ell$.
\end{defn}

When $(N,\Ds,\bK)$ is toric, $U_{\Ds}$ is a (noncompact) toric symplectic
manifold under the lifted hamiltonian $\Ab[N]$-action $\tilde\bK$ of
Observation~\ref{o:lift-ham-act}.

\begin{lemma} \label{l:toric-strata} Let $(N,\Ds,\bK)$ be a toric contact
manifold of codimension $\ell$. Then its closed orbit strata are toric contact
of codimension $\ell$, and for any $z\in N$, its stabilizer
$H:=\Stab_{\Ab[N]}(z)$ is connected \textup(i.e., a subtorus\textup). For any
$\alpha\in p^{-1}(z)$, $H=\Stab_{\Ab[N]}(\alpha)$ and the normal bundle in $N$
to the fixed point set $N^H$ is isomorphic to the symplectic isotropy
representation of $H$, which is a direct sum of $2$-dimensional symplectic
subrepresentations whose weights form a basis for the weight lattice of
$\stab_{\ab[N]}(z)^*$.
\end{lemma}
\begin{proof}[Proof cf.~\cite{Delzant,LT}] Let $m$ be the rank of $N$, $z
\in N$, $\alpha\in p^{-1}(z)$ and $\dim \stab_{\ab[N]}(z)=k$. Thus $\Ab[N]/H$
has dimension $m+\ell-k$, $U_\Ds$ has dimension $2(m+\ell)$ and the symplectic
isotropy representation $W^\alpha$ has dimension
$2(m+\ell)-2(m+\ell-k)=2k$. The nonzero weights span
$\stab_{\ab[N]}(\alpha)^*$, so they are linearly independent with
$2$-dimensional weight spaces, and $W^\alpha_0=0$. Hence the open orbit
stratum through $z$ has dimension $2(m-k)+\ell$ and its closure is toric by
Proposition~\ref{p:orbit-strata}. Since maximal tori in Sp$(2k)$ are maximal
closed abelian subgroups, $H$ is connected, and is therefore equal to
$\Stab_{\Ab[N]}(\alpha)$.
\end{proof}

Thus inverse images under $p$ of orbit strata in $N$ are orbit strata in
$U_\Ds$, and Lemma~\ref{l:sslice} specializes as follows.

\begin{prop} \label{p:local-mu} For any $\alpha\in U_\Ds$, there is
a $\Ab[N]$-invariant neighbourhood $\cU$ of the orbit $\Ab[N]\cdot\alpha\sub
U_\Ds$ on which image of the momentum map $\mu|_{\cU}-\mu(\alpha)$ is a
neighbourhood of $0$ in the product of a quadrant with
$\stab_{\ab[N]}(\alpha)^0$ in $\ab[N]^*$, it is a submersion over the interior
of this cone, and the fibres \textup(in $\cU$\textup) are $\Ab[N]$-orbits.
\end{prop}

Karshon and Lerman~\cite{KL} observe that these charts make $U_\Ds/\Ab[N]$
into a manifold with corners.  Recall (see e.g.~\cite{Joyce}) that an
$m$-dimensional \emph{manifold with corners} is a Hausdorff topological space
$Q$ equipped with a subsheaf of the continuous functions which is locally
isomorphic (as a space with a sheaf of rings) to a quadrant in a (finite
dimensional, real) vector space, hence also to $[0,\infty)^m\sub \R^m$. Here,
a function $f$ on a subset $B$ of vector space is smooth iff every $v\in B$
has an open neighbourhood $U$ such that $f|_{B\cap U}$ is the restriction of a
smooth function on $U$. Any $x\in Q$ has a neighbourhood isomorphic to a
neighbourhood of the origin in the product of a quadrant and a vector space,
i.e., to $[0,\infty)^k\times \R^{m-k}\sub \R^m$ for a well-defined $k\in\N$
called the \emph{depth} of $x$. The subspace $Q_k$ of all points with depth
$k$ is a smooth manifold of dimension $m-k$. The closures in $Q$ of the
connected components of $Q_k$ are called the $(m-k)$-dimensional faces of $Q$.
The \emph{combinatorics} $\Cb_Q$ of $Q$ is the poset of all faces, ordered by
inclusion. Faces of depth $1$ are called \emph{facets}.

In our situation, the momentum map $\mu$ also maps the fibre $p^{-1}(y)$, for
$y\in N$, to the linear subspace $\gM_y\leq \stab_{\ab[N]}(y)^0\leq \ab[N]^*$.
Since $\Stab_{\ab[N]}(y)$ fixes this fibre pointwise,
Proposition~\ref{p:local-mu} implies that $\mu(\cU)$ is foliated by its
intersection with the linear subspaces $\gM_{p(\beta)}:\beta\in\cU$. On the
other hand $\mu(\cU)$ is an open neighbourhood of $z=\mu(\alpha)$, and
shrinking $\cU$ if necessary, we may assume that the foliation of $\mu(\cU)$
is regular, so that the leaf space is isomorphic to an open neighbourhood of
zero in the product of a quadrant with $\stab_{\ab[N]}(\alpha)^0/\gM_z$.

\begin{cor}\label{c:manifoldCORNERS} The orbit space $N/\Ab[N]$ has the
structure of a manifold with corners such that the quotient map induces a
poset isomorphism from $\Cb_N$ to $\Cb_{N/\Ab[N]}$.
\end{cor}

\begin{rem}\label{r:sslice}
We shall also need to know that the symplectic slices in Lemma~\ref{l:sslice}
may be chosen compatible with the foliation of $U_\Ds$ over $M$. For this
recall that the Symplectic Slice Theorem proceeds by constructing a
differentiable slice to $\Ab[N]\cdot\alpha$ in $U_\Ds$ and then using the
Equivariant Relative Darboux Theorem to standardize the symplectic form on
$\cY=\Ab[N]\times_{\Stab_{\Ab[N]}(\alpha)} (W^\alpha\dsum\h_\alpha^0)$. Now although
the $\Ab[N]$ orbit of the fibre over $z\in N$ is $\mu_\chi^{-1}(\gM_z)$, with
$\gM_z\leq \h_\alpha^0$, there is no reason to suppose a priori that the
foliation of this orbit over $M$ is compatible with the fibration over
$\Ab[N]\cdot\alpha\cong \Ab[N]/\Stab_{\Ab[N]}(\alpha)$.

To establish this, we use the transitive action of equivariant
symplectomorphisms on angular coordinates. Since we have not found a local
proof of this in the literature, we sketch it here, using arguments
from~\cite{Weinstein}.  First, we recall from Lemma~\ref{l:toric-strata} that
$W^\alpha$ is a direct sum of $2$-dimensional weight spaces $W_1,\ldots W_k$
where $0\leq k=\dim\h_\alpha\leq m$.  Compatible complex structures on each
$W_j$ yield $\Ab[N]$-invariant angular coordinates $\theta_1,\ldots \theta_k$
on $\cY^0=\{[g,w_1+\cdots w_k+\xi]\in \cY\st w_i\neq 0\}$. The level surfaces
of $\theta_1,\ldots \theta_k$ form a $\Ab[N]$-invariant coisotropic foliation
of $\cY^0$.  Extending by angular coordinates on $\Ab[N]\cdot\alpha$ (pulled
back to $\cY^0$, we obtain an $\Ab[N]$-invariant lagrangian foliation $\cF_0$
of $\cY^0$, singular along the special orbits $w_i=0$. The fibres of $U_\Ds\to
M$ need not lie in these leaves, but they do lie in the coisotropic
foliation. Thus the coisotropic foliation contains an $\Ab[N]$-invariant
lagrangian foliation $\cF_1$ containing the fibres of $U_\Ds$ and tangent to
$\cF_0$ along the singular orbits. On a neighbourhood of the zero section of
$\cY$, there is therefore an $\Ab[N]$-equivariant diffeomorphism, equal to the
identity along the zero section, pulling back $\cF_1$ to $\cF_0$. The pullback
of $\Omega_\chi$ is thus a $\Ab[N]$-invariant symplectic form $\Omega_1$
agreeing with $\Omega_\chi$ along the zero section, and $\cF_0$ is lagrangian
for both $\Omega_1$ and $\Omega_\chi$.

We now apply the Equivariant Moser Lemma, using the observation that for
equivariant differential forms on the total space of an equivariant vector
bundle, the homotopy operator can be defined by
\[
I(\beta) = \int_0^1 \sigma_t^*(\iota_{X_t}\beta)\d t,
\]
where $\sigma_t$ denotes fibrewise scalar-multiplication by $t$ and $X_t$ is
its derivative (the radial vector field). Now, following~\cite{Weinstein}, we
set $\beta=\Omega_1-\Omega_\chi$ and $\Omega_t=\Omega_\chi + t\beta$, noting
that $\Omega_t$ is symplectic for $t\in[0,1]$ and that $\beta=\d\alpha$ with
$\alpha=I(\beta)$, so we may define $Y_t$ by $\alpha=-\iota_{Y_t}\Omega_t$;
the local flow $\phi_t$ then satisfies that $\phi_t^*\Omega_t$ is independent
of $t$, so that $\phi_1^*\Omega_1 = \phi_0^*\Omega_0=\Omega_\chi$. However for
any $\xi$ tangent to $\cF_0$,
\[
\Omega_t(Y_t,\xi) = -\alpha(\xi)
= -\int_0^1 \sigma_t^*(\Omega_1-\Omega_\chi)(X_t,\xi)=0
\]
since $X_t$ is radial, hence also tangent to $\cF_0$, and $\cF_0$ is isotropic
with respect to both $\Omega_1$ and $\Omega_\chi$. Indeed it is \emph{maximal}
isotropic, and so it follows that $Y_t$ is tangent to $\cF_0$ for all
$t$. Hence $\phi_t$ preserves $\cF_0$. We now have a symplectic slice for
$U_\Ds$ along $\Ab[N]\cdot\alpha$ in which the fibre over $z$ is contained in
$\cF_0\cap\mu_\chi^{-1}(\gM_z)$, hence in a fibre of $\cY$ over
$\Ab[N]\cdot\alpha$.
\end{rem}

\subsection{The toric symplectic cone and Levi pairs}\label{s:tsc}

We combine the local theory of the previous section with
Proposition~\ref{p:convex} to establish global properties of toric contact
manifolds of Reeb type.  We use the setting of Strategem~\ref{str:setup} in
which pairs $(\g,\lamc)$ are given by linear maps $\Ll\colon\ab[N]\to\torh$
with $\d\circ\Ll=\ul\colon\ab[N]\to\tor$.  Recall that a convex polytope
$\Pol$ in the $m$-dimensional affine space
$\As=(\afs\transp)^{-1}(1)\sub\torh^*$ is a subset of the form
\begin{equation*}
\Pol:=\{\xi\in \As\st \forall s\in\cS,\enskip L_s(\xi)\geq 0\}
\end{equation*}
where $\cS$ is a finite set, and $L_s\in\torh$ (an affine function on $\As$)
for each $s\in \cS$.

\begin{defn} Given $\Pol\sub\As$ as above, and $\xi\in \As$, let $S_\xi=
\{s\in\cS\st L_s(\xi)=0\}$ and $\Cb_\Pol=\{S_\xi\sub\cS \st \xi\in\Pol\}$ with
the induced partial ordering. We assume $\emptyset\in\Cb_\Pol$ (so
$\Pol\sub\As$ has nonempty interior) and all singletons $\{s\}:s\in\cS$ belong
to $\Cb_\Pol$ (otherwise we may discard $s$ without changing $\Pol$). The
poset $\Cb_\Pol$ is then called the \emph{combinatorics} of $\Pol$; it is
isomorphic to the poset of closed faces of $\Fa$ of $\Pol$ via the map sending
$S\in\Cb_\Pol$ to $\Fa_S=\{\xi\in\Pol\st S\sub S_\xi\}=\{\xi\in\Pol\st \forall
s\in S,\enskip L_s(\xi)=0\}$---in particular $S_\emptyset=\Pol$, and $s\in\cS$
may be identified with the facet (codimension one face)
$\Fa_s:=\Fa_{\{s\}}$. Any closed face is the intersection of the facets
containing it: $\Fa_S=\bigcap_{s\in S}\Fa_s$. We say $\Pol$ is \emph{simple}
if every vertex is $m$-valent (or equivalently $\Cb_\Pol$ is a simplicial set:
if $S'\sub S\in\Cb_\Pol$ then $S'\in\Cb_\Pol$). In this case $\Pol$ is a
manifold with corners.
\end{defn}

Recall also from Definition~\ref{d:Ncomb} that the combinatorics of $N$ is the
poset $\Cb_N$ of closed orbit strata of $N$. The analogue of a facet is a
closed orbit stratum stabilized by a circle, so we take $\cS$ to be (in
bijection with) this subset of $\Cb_N$, i.e., for each $s\in\cS$,
$N_s\in\Cb_N$ is a connected component of $N^{H_s}$ where
$H_s=\{\exp(t\e_s):t\in\R\} \leq \Ab[N]$ is a circle with primitive generators
$\pm\e_s\in\Lam\sub\ab[N]$.

\begin{defn} Let $(N,\Ds,\bK)$ be a connected toric contact manifold. For any
nondegenerate Levi pair $(\g,\lamc)$, we denote by $U^{(\g,\lamc)}$ the
connected component of $U_\Ds$ containing $\eta^\lamc(N)$.  We say $U\sub
U_\Ds$ is a \emph{Reeb component} of $U_\Ds$ if $U=U^{(\g,\lamc)}$ for some
nondegenerate Levi pair $(\g,\lamc)$.
\end{defn}
Clearly a Reeb component exists if and only if the contact action has Reeb
type.

\begin{thm} \label{t:polytope} Suppose $(N,\Ds,\bK)$ is a
\textup(compact, connected\textup) toric contact manifold, and let
$N_s:s\in\cS$ index the closed orbit strata stabilized by a circle.
\begin{numlist}
\item For any Reeb component $U$ of $U_\Ds$, the signs of the primitive
generators $\e_s$ may be chosen uniquely so that for all $s\in\cS$,
$\ip{\mu,\e_s}\geq 0$ on $U$. Furthermore, for any $\alpha\in U$, the $\e_s$
with $\ip{\mu(\alpha),\e_s}=0$ are linearly independent elements of
$\gM_{p(\alpha)}^0\leq\ab[N]$.
\item If $U=U^{(\g,\lamc)}$ for a nondegenerate Levi pair $(\g,\lamc)$, then
the image of the horizontal momentum map $\mu^\lamc\colon N\to\As\sub\torh^*$
is the compact simple convex polytope $\Pol=\Pol_{\g,\lamc}$ in $\As$ defined by
the affine functions $L_s:=\Ll(\e_s)$.  Furthermore the fibres of $\mu^\lamc$ are
$\Ab[N]$-orbits, and if $\mu^\lamc(z)$ is in the interior of a face $\Fa$ then
the annihilator of the tangent space to $\Fa$ at $z$ is
$\ul(\stab_{\ab[N]}(z))\leq\tor \cong T^*_{\mu^\lamc(z)}\As$.

In particular, $\mu^\lamc$ induces a poset isomorphism of $\Cb_N$ with
$\Cb_\Pol$, sending $N_s$ to $\Fa_s$ for all $s\in\cS$ \textup(so both $\Cb_N$
and $\Cb_\Pol$ are simplicial sets\textup), and the restriction of $\mu^\lamc$
to an orbit stratum is a submersion over the interior of the corresponding
face.
\end{numlist}
\end{thm}
\begin{proof} First note that by Lemma~\ref{l:toric-strata}, the orbit strata
have connected stabilizers, so $N_s:s\in\cS$ are the maximal proper strata in
$\Cb_N$; furthermore $\ip{\mu(N_s),\e_s}=0$.

Proposition~\ref{p:local-mu} shows that the connected components of the fibres
of $\mu\colon U_\Ds\to \ab[N]^*$ are $\Ab[N]$ orbits, and if $\alpha\in N_s$
with stabilizer $\mathopen<\exp(t\e_s)\mathclose>$, then $\mu$ maps a
neighbourhood of $\Ab[N]\cdot\alpha$ to a half space bounded by
$\ip{\mu,\e_s}=0$.

Given a nondegenerate Levi pair $(\g,\lamc)$, Proposition~\ref{p:convex}
implies, by choosing a basis for a complement to the image of $\ker\lamc$ in
$\ab[N]$, that the image $\Pol$ of $\mu^\lamc$ is the convex hull of the
points $\mu^\lamc(z)$ where $\Rb\g_z=\Rb{}_z$, and that the fibres of
$\mu^\lamc$ are connected. Since $\Ll\transp\circ\mu^\lamc=\mu\circ\eta^\lamc$
is the restriction of $\mu$ to $\eta^\lamc(N)={\mu_\g}^{-1}(\lamc)$, the
fibres of $\mu^\lamc$ are $\Ab[N]$-orbits and we may choose the signs of
$\e_s$ so that $L_s:=\Ll(\e_s)\geq 0$ on $\Pol$.

Since $\mu$ maps the fibres of $U_\Ds\sub \Ds^0\to N$ linearly onto $\ab[N]^*$
(sending $\Ds^0_z$ to $\gM_z\leq\ab[N]^*$) these signs ensure
$\ip{\mu,\e_s}\geq 0$ on the Reeb component containing $\eta^\lamc(N)$,
proving (i).

The explicit local description of Proposition~\ref{p:local-mu} now shows that
$\Pol$ is the compact convex polytope on which $L_s\geq 0$ and that this
polytope is simple with the stated face and submersive properties, proving
(ii). The last part now follows.
\end{proof}

We may assume, if necessary, that the primitive generators $\e_s:s\in\cS$ span
$\ab[N]$: otherwise, there is a nontrivial subtorus $H$ of $\Ab[N]$ acting
freely on $N$ and transversely to $\Ds$. Thus $N$ is a principal $H$-bundle
over a toric contact manifold $(N/H, \Ds/H)$. Conversely, any principal torus
bundle with connection over a toric contact manifold is toric contact as well.
There is no reason to suppose that $\e_s:s\in\cS$ are linearly dependent, or
even distinct, but the case that they form a basis is of particular interest.

\begin{ex}\label{e:contactLabPol} Theorem~\ref{t:polytope} associates to
every Reeb component of a toric contact manifold, a simple convex
polytope
\begin{equation*}
\Pol:=\{\xi\in \As=(\afs\transp)^{-1}(1)\sub\h^*
\st \forall s\in\cS,\enskip L_s(\xi)\geq 0\}
\end{equation*}
which is \emph{labelled} in the sense that the affine functions $L_s\in\h$,
parametrized by the set $\cS$ of facets of $\Pol$, are given, whereas in
general they are only determined by $\Pol$ up to rescaling by a positive real
number. Let $\Z_\cS$ be the free abelian group generated by $\cS$, let
$\ab=\Z_\cS\otimes_\Z \R$ and $\C_\cS=\Z_\cS\otimes_\Z \C$ be the
corresponding free vector spaces over $\R$ and $\C$, and let
$\Ab=\ab/2\pi\Z_\cS$.  Denote the generators of $\Z_\cS\sub\ab\sub \C_\cS$ by
$e_s:s\in \cS$. The affine functions $L_s$ are encoded by the induced linear
map $\Ll\colon\ab\to \h$ with $\Ll(e_s)=L_s$.

Conversely, given a simple labelled polytope $(\Pol,\Ll)$, we now exhibit
many toric contact manifolds with a Reeb component inducing these data.  For
this, we first observe that $\Ab$ and $\Ab^\C= \C_\cS/2\pi\Z_\cS\cong
\C_\cS\mult$ act diagonally on $\C_\cS$, via $[\sum_s t_s e_s]\cdot (\sum_s
z_s e_s)=\sum_s\exp (i t_s) z_s e_s$. The action of $\Ab$ on $\C_\cS$ is
hamiltonian with respect to the standard symplectic form $\omega_\cS$ on $\cS$
and has a momentum map $\mm\colon\C_\cS\to \ab^*$ defined by
\begin{equation}\label{eq:std-momentum}
\ip{\mm(z),e_s}=\sigma_s(z)=\tfrac{1}{2}|z_s|^2,
\end{equation}
where $z_s\colon \C_\cS\to \C$ denote the standard (linear) complex
coordinates on $\C_\cS$. We define $\g = \ker (\d\circ\Ll\colon \ab\to\tor)$
and let $\iota_\g\colon \g \into \ab$ denote the inclusion. Since
$\Ll\circ\iota_\g$ takes values in $\ker\d$ there is a unique $\lamc\in\g^*$
with $\Ll\circ \iota_\g=\afs\circ\lamc$. We now define
\[
N:=(F\circ\mm)^{-1}(0) \quad \text{with}\quad
F=\iota_\g\transp-\lamc\colon \ab^*\to \g^*.
\]
Since $\Pol$ is simple, $\lamc$ is a regular value of $\iota_\g\transp \circ \mm$
which is a momentum map for the action of $\g$. Thus we get that $\Rb\g$ has
full rank everywhere on $N$. Denote by $\tau_\cS$, a primitive of $\omega_\cS$
such that $\tau_\cS (K_w)= \langle \mm, w \rangle$ for every $w\in\ab$ (for
e.g. take $\tau_\cS =\sum_{s\in \cS} d^c\mm_s$). By definition, $N
=\{z\in\C_\cS \,|\, \langle \mm_z, w\rangle = \langle \lamc, w\rangle \,
\forall w\in\g\}$. Hence, on $N$, for any $w\in\g$ we have $\tau_\cS(K_w) =
\langle \lamc, w\rangle$ and $\Rb{\ker \lamc} \sub \ker \tau_\cS$.  Observe
that the convexity of $\Pol$ implies that $\g \cap (\R_{>0})_\cS$ is not
empty. In particular, there exists $w\in\g$ such that $\tau_\cS(K_w)=
\sum_{s\in\cS} w_s|z_s|^2 \neq 0$ which ensures that the rank of $(\ker
\tau_\cS) \cap TN$ is $2m+\ell -1$.  Any complementary bundle $\Ds$ of
$\Rb{\ker \lamc}$ in $(\ker \tau_\cS) \cap TN$ is contact and admits
$(\g,\lamc)$ as a Levi pair. Indeed, for any local sections $X,Y \in
\Gamma(\Ds)$, where $\Ds\sub \ker\tau_\cS$ we have
\[
\omega(X,Y) = \d\tau_\cS(X,Y) = -\tau_\cS([X,Y]),
\]
whereas $\omega$ is nondegenerate on $TN/\Rb\g$ (by symplectic reduction) and
thus on $\Ds$.

If $\Ds$ is $\Ab$-invariant we get a contact toric manifold $(N,\Ds,\bK)$
together with a Levi pair $(\g,\lamc)$ for which the pull back of $\tau_\cS$
is the contact form $\eta^\lamc$. Recall that the momentum map $\mu$ of $\Ab$
at $\alpha\in \Ds^0$ is defined by $\langle \mu_\alpha, v\rangle =
\alpha(K_v)$ for all $v\in \ab^*$. Thus $\langle \Ll\transp\circ\mu^\lamc ,
v\rangle = \langle\mu\circ \tau_\cS , v\rangle = \sum_{s\in\cS} v_s|z_s|^2$
and $\Ll\transp(\im \mu^\lamc) = (\iota_\g\transp)^{-1}(\lambda)\cap\im \mm$
coincides with $\Ll\transp(\Pol)$. Such contact distributions exist---for
instance $\Ds=(\ker \tau_\cS)\cap \big(\bigcap_{w\in \ker \lamc}\ker
(d^c\langle \mm, w\rangle) \big)$ works.
\end{ex}

In this example, there is a close connection with the Delzant--Lerman--Tolman
construction of toric symplectic orbifolds. More generally, we have the
following.

\begin{rem}\label{r:stabG} If $(\g,\lamc)$ be a nondegenerate Levi pair
for $(N,\Ds,\bK)$ then there is a subtorus $G\leq \Ab[N]$ with Lie algebra
$\g$ if and only if the lattice of circle subgroups in $\ab[N]$ is mapped, via
$\ul$, to a lattice $\Lam$ in $\tor$ containing $\spns_\Z\{u_s=\ul(\e_s) \st
s\in \cS\}$. In this case, the quotient $(N/G, d\eta^\lamc\restr{\Ds},
\Ab[N]/G)$ is a toric symplectic orbifold, and $\mu^\lamc$ descends to a
momentum map on $N/G$. If $[z]\in N/G$ such that $\mu^{\lamc}(z)\in \Fa_S$ for
some $S\sub\cS$ in $\Cb_\Pol$, then the stabilizer of $z$ in
$\Ab[N]=\ab[N]/2\pi\Lam$ is connected and $\stab_{\ab[N]}(z)= \spns_\R \{
\e_s\st s\in S\}$.  An element $[v]=[\sum_{s\in S}t_s\e_s]$ of the stabilizer
of $z$ in $\Ab[N]$ is in $G$ iff $\sum_{s\in S} t_su_s \in\Lam$, and is zero
iff $t_s\in\Z$ for all $s\in S$.  Hence the orbifold structure group of $[z]$
\textup(the stabilizer of $z$ in $G$\textup) is isomorphic to
$(\Lam\cap\spns_\R\{u_s\st s\in S\})/\spns_\Z \{u_s\st s\in S\}$.

The Delzant--Lerman--Tolman correspondence then associates the symplectic
toric orbifold $N/G$ with $(\Pol,\Ll)$. However, if $\g$ is not the Lie
algebra of a closed subtorus of $\Ab[N]$, the image of $\Lam$, the lattice of
circle subgroups in $\ab[N]$, is not sent to a lattice in $\tor$ via $\ul$
and there is no preferred lattice in $\tor$. In particular, there is no
reason for the polytope $\Pol$ to be rational with respect to any lattice.
\end{rem}

\subsection{The grassmannian image}\label{s:gr-im}

Let $(N,\Ds,\bK)$ be a \textup(compact, connected\textup) codimension $\ell$
toric contact manifold of Reeb type, and let $N/\Ab[N]$.
Corollary~\ref{c:manifoldCORNERS} and Theorem~\ref{t:polytope} imply that
$N/\Ab[N]$ is a compact connected manifold with corners of dimension $m$,
diffeomorphic to a simple polytope $\Pol$. However,
Example~\ref{e:contactLabPol} shows that even the \emph{labelled} polytope
$(\Pol,\Ll)$ is insufficient data to recover $(N,\Ds,\bK)$.

To remedy this we consider instead the grassmannian momentum map
\begin{align*}
\gM \colon &N \to \Gr_\ell(\ab[N]^*)\\
           &z \mapsto \gM_z=\mu(\Ds^0_z)
\end{align*}
introduced in Definition~\ref{d:gmm}: recall that if $(N,\Ds,\bK)$ is locally
Reeb type, then $\mu$ maps the fibres of $\Ds^0$ to $\ell$-dimensional linear
subspaces of $\ab[N]^*$. If $(\g,\lamc)$ is a nondegenerate Levi pair, then by
Remark~\ref{r:affine}, $\gM$ takes values in $\Gr_\ell(\ab[N]^*)$ and
$\ip{\psi_\g\circ\gM,\lamc}=\Ll\transp\circ\mu^\lamc$. Since the latter is an
orbit map and $\gM$ is $\Ab[N]$-invariant, $\gM$ induces an embedding of
$N/\Ab[N]$ into $\Gr_\ell(\ab[N]^*)$. This is the \emph{grassmannian image}
that we will use to characterize $(N,\Ds,\bK)$.

\begin{defn}\label{d:polyhedral} Let $\ab[]$ be the Lie algebra of a
compact $(m+\ell)$-torus $\T= \ab[]/2\pi\Lam$. We say that a manifold with
corners $\Xi\sub\Gr_\ell(\ab[]^*)$, with facets $\Xi_s:s\in\cS$, is
\begin{bulletlist}
\item \emph{polyhedral}, if for all $s\in\cS$, there is a $1$-dimensional
subspace $E_s\leq\ab[]$ such that the facet $\Xi_s$ lies in the
subgrassmannian $\Gr_\ell(E_s^0)\sub \Gr_\ell(\ab[]^*)$;
\item \emph{labelled} by $e=(e_s)_{s\in\cS}$ if for all $s\in \cS$,
$e_s\in\ab[]\setminus 0$ and $\Xi_s\sub \Gr_\ell(e_s^0)$.
\end{bulletlist}
A labelling $e$ of $\Xi$ determines a (possibly empty) cone
\begin{equation}\label{defnCONExie}
\cC_e:= \{ x\in \ab[]^* \st \forall\,s\in \cS,\; \ip{x, e_s} \geq 0
\text{ and } S_x\in \Cb_\Xi\},
\end{equation}
where $S_x=\{s\in \cS\st \ip{x, e_s}=0\}$. A labelled polyhedral manifold with
corners $(\Xi,e)$ is
\begin{bulletlist}
\item \emph{rational} if $e_s \in \Lam$ for all $s\in \cS$;
\item \emph{Delzant} if for each $S\in \Cb_\Xi$, $(e_s)_{s\in S}$ is a
$\Z$-basis for $\Lam\cap\mbox{span}_{\R}\{ e_s\,|\, s\in S\}$.
\item \emph{of Reeb type} if there is an $\ell$-dimensional subspace
  $\iota_\g\colon\g\hookrightarrow \ab[]$ and $\lamc\in \g^*$ such that
  $\Xi\sub\Gr_\ell^\g(\ab[]^*)$ and $\ip{\psi_\g,\lamc}\colon
  \Gr_\ell^\g(\ab[]^*)\to \ab[]^*$ maps $\Xi$ bijectively to the convex
  polytope
\begin{equation}\label{defnPOLxie}
\Pol_{\g,\lamc}:=(\iota_\g\transp)^{-1}(\lamc)\cap \cC_e.
\end{equation}
\end{bulletlist}
\end{defn}
Note that the Delzant and Reeb type conditions for $(\g,\lamc)$ determine the
labelling $e$.

\begin{cor}\label{c:assMnldCorners}
The grassmannian image $\Xi$ of a compact toric contact manifold of Reeb type
$(N,\Ds,\bK)$ is polyhedral in $\Gr_\ell(\ab[N]^*)$, and any Reeb component
$U\sub U_\Ds$ induces a labelling $e$ such that $(\Xi,e)$ is Delzant of Reeb
type.
\end{cor}

We would like to show that the grassmannian image classifies toric contact
manifolds of Reeb type, following a well-known line of argument~\cite{KL,LT}
in the symplectic and codimension one case, which builds on~\cite{HS}.
However, in higher codimension, local invariants obstruct a straightforward
description of $\con(N,\Ds)$, and so the method yields less explicit
conclusions. Let $\Con^\T_0(\Ds)$ be the sheaf over $N/\Ab[N]$ of
$\Ab[N]$-equivariant contactomorphisms of $N$ preserving the orbits (i.e.,
inducing the identity on $N/\Ab[N]$) and let $\con^\T(\Ds)$ be the sheaf over
$N/\Ab[N]$ of $\Ab[N]$-invariant contact vector fields on $(N,\Ds)$.

Note that if a contact vector field $X$ is $\Ab[N]$-invariant, then so is the
hamiltonian $f_X$ for its lift $\tilde X$ to $U_{\Ds}$, i.e.,
$\Omega^\Ds(\tilde X,\tilde K_v)=-\d f_X(\tilde K_v)=0$ for any $v\in\ab[N]$.
Since the latter span a lagrangian subspace on a dense open set, $\tilde
X_\alpha = \tilde K_{\tilde v(\alpha),\alpha}$ for some $\ab[N]$-valued
function of the form $\tilde v= p^*v$ (since $\tilde X$ is
$\ab[N]$-invariant). Hence for any $z\in N$, $X_z=K_{v(z),z}$ is tangent to
the $\Ab[N]$-orbit through $z$.

\begin{lemma} There is an exact sequence of sheaves over $N/\Ab[N]$:
\begin{equation}\label{eq:sess}
0\to\underline{2\pi\Lam}\to \con^\T(\Ds)\to \Con^\T_0(\Ds)\to 0,
\end{equation}
where the first nontrivial map is the inclusion of the locally constant sheaf
associated to $2\pi\Lam\sub\ab[N]$ into $\con^\T(\Ds)$ and the second is the
time $1$ flow of a contact vector field.
\end{lemma}
\begin{proof} Exactness at the first two steps is straightforward,
following~\cite{KL,LT}; surjectivity of the second map is less so.  However,
given a contactomorphism $\phi$ in $\Con^\T_0(U,\Ds)$ over a
$\Ab[N]$-invariant open subset $U$ of $N$, the lift $\tilde \phi$ to
$p^{-1}(U)$ preserves $\tau^\Ds$, hence the momentum map $\mu$, hence the
$\Ab[N]$-orbits.  Thus~\cite[Theorem~3.1]{HS} shows that for $U$ sufficiently
small, $\tilde\phi(\alpha)=\exp(\tilde v(\alpha))\cdot\alpha$ for a smooth
$\Ab[N]$-invariant function $\tilde v\colon p^{-1}(U)\to \ab[N]$.  Since
$p(\tilde\phi(\alpha))=\phi(p(\alpha))$, $\exp(\tilde v(\alpha))\cdot
p(\alpha)=\phi(p(\alpha))$ and so we may assume $\tilde v = p^*(v)$ for a
smooth $\Ab[N]$-invariant function $v\colon U \to\ab[N]$. Thus
$\tilde\phi=\psi_1$, where $\psi_t$ is the flow of the projectable vector
field $\tilde X_\alpha=\tilde K_{v(p(\alpha)),\alpha}$ (this flow exists
because $\tilde X$ is tangent to the compact $\Ab[N]$ orbits).

We now show that $\tilde X$ is the lift to $p^{-1}(U)$ of a infinitesimal
contactomorphism. First note that since the $\Ab[N]$-orbits are isotropic,
$\Omega^\Ds(\tilde X, K_w)=0$ for all $w\in\ab[N]$ and hence $\cL_{\tilde
  X}\tau^\Ds = \iota_{\tilde X} \Omega^\Ds + \d \iota_{\tilde X}\tau^\Ds$ is a
basic $1$-form with respect to the $\Ab[N]$ action.  Since $\psi_1=\tilde\psi$
and $\psi_0=\Id$ preserve $\tau^\Ds$,
\begin{equation*}
0=\psi_1^*\tau^\Ds-\psi_0^*\tau^\Ds
=\int_0^1 \frac{\d}{\d t} \psi_t^*\tau^\Ds \, \d t
=\int_0^1 \psi_t^* (\cL_{\tilde X} \tau^\Ds) \, \d t = \int_0^1
\cL_{\tilde X} \tau^\Ds \,\d t = \cL_{\tilde X} \tau^\Ds,
\end{equation*}
where in the penultimate step, we have used that $\psi_t$ induces the identity
on the $p^{-1}(U)/\Ab[N]$, and so preserves basic $1$-forms.
\end{proof}

Identifying $N/\Ab[N]$ with $\Xi$ via the grassmannian momentum map, we may
view~\eqref{eq:sess} as an exact sequence of sheaves on $\Xi$.

\begin{prop} Let $(N,\Ds,\bK)$ be a toric contact manifold of Reeb type with
grassmannian image $\Xi$. Then toric contact manifolds of Reeb type with the
same grassmannian image are parametrized up to isomorphism by
$H^1(\Xi,\con^\T(\Ds))$.
\end{prop}
\begin{proof} Suppose $N'$ has the same grassmannian image. This induces a
diffeomorphism between $N/\Ab[]$ and $N'/\Ab[]$. It follows from
Remark~\ref{r:sslice} that $N$ and $N'$ are locally isomorphic by a fibre
preserving contactomorphism. Thus a standard argument~\cite{HS,KL,LT} shows
that $N'$ determines and is determined up to isomorphism by an element of
$H^1(\Xi,\Con^\T_0(\Ds))$. However, since $N$ has Reeb type, $\Xi$ is
diffeomorphic to a simple convex polytope, and hence is contractible, so
$H^i(\Xi,\underline{2\pi\Lam})=0$ for $i\geq 1$. Thus by the long exact
sequence associated to~\eqref{eq:sess}, $H^1(\Xi,\Con^\T_0(\Ds))\cong
H^1(\Xi,\con^\T(\Ds))$.
\end{proof}

In order to understand the sheaf $\con^\T(\Ds)$, it is convenient to introduce
a transversal subalgebra $\g\leq \ab[N]$ which is the Lie algebra of a
subtorus $G\leq \Ab[N]$ and has $\cC_\g\sub\g^*$ nonempty. Thus $N$ is a
principal $G$-bundle over a compact orbifold $M=N/G$, and $\Ds$ defines a
$\Ab[M]:=\Ab[N]/G$-invariant principal $G$-connection $\eta$ on $N\to M$,
cf.~Example~\ref{con-pb}.

\begin{defn} A \emph{toric $\ell$-symplectic manifold} (or \emph{orbifold}) is
a $2m$-manifold (or orbifold) $M$ with an effective action of an $m$-torus
$\Ab[M]$ together with a $\Ab[M]$-invariant principal $G$-bundle $\pi\colon
N\to M$ with connection $\eta$, for $G$ an $\ell$-torus, such that for some
$\lamc\in\g^*$, $\ip{\omega,\lamc}$ is nondegenerate, where $\omega\in
\Omega^2(M,\g)$ is the curvature of $\eta$.
\end{defn}
In this setting, $(N,\Ds=\ker\eta)$ is a toric contact manifold (or orbifold)
of dimension $\ell$ under an extension $\Ab[N]$ of $\Ab[M]$ by $G$, the
induced action of $\g$ is transversal, and $\ip{\Omega,\lamc}$ is
nondegenerate (hence a symplectic form on $M$) iff $\lamc\in \cC_\g$, i.e.,
$(\g,\lamc)$ is a Levi pair.

\begin{prop} Let $M$ be a toric $\ell$-symplectic manifold with associated
toric contact manifold $(N,\Ds,\Ab[N])$. For $Z\in\con^\T(\Ds)$, viewed as a
sheaf on $N$, we may write $\eta(Z)=\pi^* f_Z$ and this defines an
isomorphism between $\con^\T(\Ds)$ and the sheaf of $\Ab[M]$-invariant
$\g$-valued functions $f$ on $M$ such that $\d f = -\iota_X \omega$ for some
vector field $X$.
\end{prop}
\begin{proof} For $Z\in\con^\T(\Ds)$, $\eta(Z)$ is $G$-invariant, hence of
the form $\pi^*f_Z$ for some $\g$-valued function $f_Z$ on $M$, which is
$\Ab[M]$-invariant because $\eta(Z)$ is $\Ab[N]$-invariant. However, since $Z$
is contact $0=\cL_Z\eta=\iota_Z\d\eta + \pi^*\d f_Z$, so if $f_Z=0$ on an open
subset of $M$, $\iota_Z\d\eta=0$ on its inverse image in $N$, and hence $Z=0$
(since $\d\eta$ is nondegenerate on the kernel of $\eta$). Since
$\iota_Z\d\eta$ is $G$-invariant, and vanishes on generators of the
$G$-action, it is basic. However $\d\eta=\pi^*\omega$, so $Z$ is projectable,
to a vector field $X$ with $\iota_X\omega=-\d f_Z$. Conversely, for any
$\Ab[M]$-invariant $\g$-valued functions $f$ with $\d f = -\iota_X \omega$ for
some vector field $X$, we may take $Z= X^H+\ip{\pi^* f,\bK^\g}$, where $X^H$
is the horizontal lift of $X$ to $\Ds$ and $\bK^\g\colon\g\to\con^\T(\Ds)$
generate the infinitesimal action of $\g$ on $N$. Then $Z$ is
$\Ab[N]$-invariant with $\eta(Z)=\pi^*f$, and $\cL_Z\eta=\iota_Z\d\eta +
\pi^*\d f_Z=\pi^*(\iota_X\omega+\d f) = 0$, so $Z$ is contact.
\end{proof}

Note that the vector field $X$ here is $\ell$-hamiltonian, in the sense that
it has a hamiltonian with respect to every (nondegenerate) component of
$\omega$. When $\ell=1$, or more generally, when the Levi form has rank one
image (the $\ell$-contact manifolds
of~\cite{Bolle,Dragnev}---cf.~Examples~\ref{ex:contact} (iii)), this places no
condition on $f$, so $\con^\T(\Ds)$ is isomorphic to the sheaf of functions on
$M/\Ab[M]=\Xi$, which is a fine sheaf, so $H^1(\Xi,\con^\T(\Ds))=0$.

We also have the following generalization.

\begin{prop}\label{p:sheaf-prod} If $N$ is a product of codimension one contact
  manifolds of Reeb type, then $H^1(\Xi,\con^\T(\Ds))=0$.
\end{prop}
\begin{proof} We take $G$ to be the product of transverse actions on each
factor, so that $M=N/G$ is a product of symplectic manifolds
$(M_i,\omega_i):i\in\{1,\ldots \ell\}$ and
$\omega=\omega_1\oplus\cdots\oplus\omega_\ell\in \Omega^2(M,\R^\ell)$.  Now
for any vector field $X$ on $M$, $-\iota_X\omega=(\alpha_1,\ldots
\alpha_\ell)$ where $\alpha_i$ is a section of the pullback of $T^*M_i$ to
$M$, viewed as a subbundle of $T^*M$.  For this to be of the form $(\d
f_1,\ldots \d f_\ell)$ for functions $f_1,\ldots f_\ell$ on $M$, we must have
$\d f_i(X_j)=0$ whenever $X_j$ is tangent to $M_j$ for $j\neq i$. Hence
$f_1,\ldots f_\ell$ are locally pullbacks from $M_1,\ldots M_\ell$, and are
otherwise constrained only by $\Ab[N]$-invariance.  Thus $\con^\T(\Ds)$ is a
product of pullbacks of fine sheaves, so has $H^1(\Xi,\con^\T(\Ds))=0$.
\end{proof}
This result applies in particular to products of spheres.

\subsection{Existence}\label{s:exists}

It remains to construct a toric contact manifold of Reeb type from a Delzant,
Reeb type, labelled polyhedral manifold with corners $(\Xi,e)$ in
$\Gr_\ell(\ab[]^*)$.

\begin{thm}\label{t:constructionN}  Let $\ab[]$ be the Lie algebra of
a $(\ell+m)$-dimensional torus $\T=\ab[]/2\pi\Lam$ and $(\Xi,e)$ be a Delzant,
Reeb type, labelled polyhedral manifold with corners in $\Gr_\ell(\ab[]^*)$.
There is a codimension $\ell$ compact toric contact manifold with a Reeb
component whose associated labelled manifold with corners is $(\Xi,e)$.
\end{thm}

\begin{proof} The Reeb type condition implies there exist $\iota_\g\colon
\g \hookrightarrow \ab[]$ and $\lamc \in \g^*$ such that
$\ip{\psi_\g,\lamc}\colon \Xi\to \Pol_{\g,\lamc}$ is a diffeomorphism, with
$\Pol_{\g,\lamc}$ defined by \eqref{defnPOLxie}. This condition is open in
$\lamc$ (by compactness of $\Xi$, cf.~Remark~\ref{r:open}), so for $\lamc'$ in
a contractible open neighbourhood $\cN$ of $\lamc\in \g^*$,
$\ip{\psi_\g,\lamc'}$ is a diffeomorphism from $\Xi$ to
$\Pol_{\g,\lamc}$. Thus
\[
V:= (\iota_\g\transp)^{-1}(\cN) \cap\cC_e
\]
(with $\cC_e$ defined by~\eqref{defnCONExie}) is foliated by its intersections
with $\xi\in \Xi$, i.e., it has a foliation $\cF$ with leaf space $\Xi$ and
leaves $\cF_\xi:=V\cap \xi$ for $\xi\in\Xi$. Note that $\cF$ is compatible
with the face decomposition of $V$: for any $S\in\Cb_\Xi$ and any
$\xi\in\Xi_S$, $\cF_\xi\sub V_S:=\{x\in V\st \forall\, s\in
S,\;\ip{x,e_s}=0\}$.  Let $\widetilde{V}$ be open in $\ab[]^*$ with $V=
\widetilde{V}\cap \cC_e$.

We now use a variant of the Delzant construction, cf.~\cite{KL}, to build a
toric symplectic $2(m+\ell)$-manifold $(M,\omega)$ under $\T$. First, since
$\widetilde{V} \times \T\hookrightarrow \ab[]^*\times \T \simeq T^*\T$, it is
a toric symplectic manifold with momentum map
\[
x\colon \widetilde{V} \times \T \rightarrow \ab[]^*
\]
induced by the embedding of the first factor into $\ab[]^*$. It has
$\T$-invariant symplectic form $\tilde{\omega}=\ip{\d x\wedge\d\theta}$, where
$\theta\colon\T\to \ab[]/2\pi\Lam$ are the tautological angular coordinates on
$\T$.

Next, as in Example~\ref{e:contactLabPol}, let $\Z_\cS$ be the free abelian
group generated by $\cS$, let $\ab=\Z_\cS\otimes_\Z \R$ and
$\C_\cS=\Z_\cS\otimes_\Z \C$ be the corresponding free vector spaces over $\R$
and $\C$, and let $\Ab=\ab/2\pi\Z_\cS$. Since $(\Xi,e)$ is rational, the map
$s\mapsto e_s\in \Lam$ induces a group homomorphism $\Z_\cS\to\Lam$ and hence
also homomorphisms $\pi\colon\Ab\to \T$ and $\pi_*\colon \R_{\cS}\to \ab[]$.

We now consider the action of $\gamma\in\T_\cS$ on $\widetilde{V} \times
\T\times \C_\cS$ by
\begin{equation}\label{actionTS}
\gamma\cdot (x,[\theta],z) = (x,\pi(\gamma)\cdot[\theta],\gamma\cdot z)
\end{equation}
where $\T$ acts on $\T$ and $\T_\cS$ on $\C_\cS$ in the standard ways.  This
action is Hamiltonian for the symplectic form $\tilde{\omega}\oplus
(-\omega_{\mathrm{std}})$ on $(\widetilde{V} \times \T)\times \C_\cS$ with
momentum map
\begin{equation}
\begin{split}
\phi\colon  \widetilde{V} \times \T\times \C_\cS & \to \R_{\cS}^*\cong\R^\cS\\
(x,[\theta],z) &\mapsto\bigl(l_s(x)-\tfrac{1}{2}|z_s|^2\bigr)_{s\in\cS}.
\end{split}
\end{equation}
where $l_s(x):= \langle x, e_s\rangle$ are the linear maps defining
$\cC_e$. We now show that $\T_\cS$ acts freely on $\phi^{-1}(0)$. Indeed, if
$x\in\widetilde{V}\backslash V$ then $\phi(x,[\theta],z)<0$ and so
\begin{equation}\label{phi-10}
\phi^{-1}(0) = \bigcup_{x\in V}\Big(\{x\}\times \T \times \T_\cS(x)\Big)
\end{equation}
where $\T_\cS(x) := \prod_{s\in\cS} \Sph^1_{l_s(x)}$ and
$\Sph^1_{l_s(x)}=\left\{z\in\C \,\big|\; \frac{1}{2}|z|^2 = l_s(x) \right\}$ is a
circle of radius $\sqrt{2 l_s(x)}\geq 0$.  Hence if
$\gamma\in\Stab_{\T_\cS}(x,[\theta],z)$ then for all $s\in\cS$, either
$\gamma_s=0$ or $l_s(x)=0$. Now for $x\in V$, $S_x:=\{s\in\cS\st l_s(x)=0\}$
is in $\Cb_\Xi$ and so $e_s:s\in S_x$ are linearly independent. Indeed by the
Delzant condition, they belong to a basis for $\Lam$, and so if
$\pi(\gamma)=0$ then $\gamma=0$.

Consequently, the symplectic reduction of $\widetilde{V} \times \T\times
\C_\cS$ at $0$ with respect to the action~\eqref{actionTS} of $\T_\cS$ is a
smooth symplectic manifold $(M,\omega)$. Observe that the natural action of
$\T$ on the second factor of $\widetilde{V} \times \T\times \C_\cS$ commutes
with the action~\eqref{actionTS} of $\T_\cS$ and descends to an effective
action on $M$, whose momentum map $\mu\colon M\rightarrow \ab[]^*$ is
determined by $\mu\circ q=x|_{\phi^{-1}(0)}$ where $q\colon \phi^{-1}(0) \to
M$ is the quotient by $\T_\cS$.

The foliation $\cF$ of $V$ induces a $\T_{\cS}$-invariant isotropic foliation
of $\phi^{-1}(0)\sub V\times \T\times \C_\cS$, which descends to a
$\T$-invariant isotropic foliation $\cG$ of $M$.  Note that
$N_{\g,\lamc}:=\mu^{-1}(\Pol_{\g,\lamc})$ is a global transversal submanifold,
i.e., meets each leaf of $\cG$ in a single point.  Moreover, $N_{\g,\lamc}$ is
smooth, compact (recall that $\mu$ is proper) and $\T$-invariant. Thus the
space of leaves $N:=M/\cG$ is $\T$-equivariantly diffeomorphic to
$N_{\g,\lamc}$.  We denote by $f\colon M\to N$ the quotient map with
$\cG_z=f^{-1}(z)$.  The momentum map $\mu\colon M\to V$ is an orbit map for
$\T$, and it descends to a $\T$-invariant map $\gM\colon N \to \Xi$, hence a
rank $\ell$ subbundle $\gM\leq N\times \ab[]^*$ over $N$ with fibre
$\gM_z=\gM(z)$ such that, for $z\in N$, we have $\mu(\cG_z) = \cF_{\gM_z}$.

On the other hand, the symplectic quotient construction provides action angle
coordinates on $\mathring{M} :=\mu^{-1}(\mathring{V})$, that is a splitting of
the sequence
\[
0 \rightarrow \mathring{M}\times \ab[] \stackrel{\kappa}{\longrightarrow}
T\mathring{M} \stackrel{d\mu}{\longrightarrow} \mathring{M}\times\ab[]^*
\rightarrow 0
\]
which identifies
\begin{equation}\label{aaCOORDsplitting}
T\mathring{M} \simeq \mathring{M}\times (\ab[]\oplus \ab[]^*)
\end{equation}
such that the symplectic structure $\omega$ restricted on $\mathring{M}$
coincides with the fibrewise pairing on the left side. Now observe that via
the identification~\eqref{aaCOORDsplitting}, the pullback of $\gM^0\oplus
\ab[]^* \leq N\times (\ab[]\oplus \ab[]^*)$ on $M$, that is $f^*(\gM^0\oplus
\ab[]^*) \rightarrow M$, coincides over $\mathring{M}$ with
$T\mathcal{G}^{\perp_\omega}$. Hence $T\mathcal{G}^{\perp_\omega}$ is
projectable on $\mathring{N}:=f(\mathring{M})$. However a distribution
containing $T\cG$ is (locally) projectable if and only if it is preserved by
Lie derivatives along vector fields in $T\cG$. Hence
$T\mathcal{G}^{\perp_\omega}$ is projectable on $M$ by continuity.

We now show that $\Ds:= f_*(T\mathcal{G}^{\perp_\omega}) \leq TN$, a rank $2m$
distribution on $N$, is contact. If we denote $v= x\frac{\partial}{\partial
  x}\in TV$ the vector field induce from the $\R^*$-action of dilation, $v$ is
obviously tangent to each leaf of $\cF$ and thus defines a vector field $Y$ on
$M$ such that $\cL_Y\omega = \d\iota_Y\omega = \omega$. In particular,
$\iota_Y\omega$ is a primitive $1$-form whose kernel contains
$T\mathcal{G}^{\perp_\omega}$, and by using the same argument as in
Example~\ref{e:contactLabPol}, we conclude that $\Ds$ is contact.
\end{proof}

\begin{rem} If we begin the construction with $(\Xi,e)$ rational instead of
Delzant, we obtain instead a toric contact orbifold.
\end{rem}

\appendix 
\section{Affine geometry of natural momentum maps}\label{a:nmm}

Let $(M,\omega)$ be a connected symplectic manifold; then the \emph{symplectic
  gradient} $\grad_\omega f$ of a smooth function $f$ is the unique vector
field $X$ with $\d f= -i_X\omega$. The image of $\grad_\omega$ is the Lie
subalgebra $\ham(M,\omega)\leq C^\infty(M,TM)$ of \emph{hamiltonian vector
  fields}, and there is an exact sequence
\begin{equation*}
0\to \R \to C^\infty(M,\R)\stackrel{\grad_\omega}{\longrightarrow}
\ham(M,\omega)\to 0
\end{equation*}
of Lie algebras, where $C^\infty(M,\R)$ (the space of smooth functions) is a
Lie algebra under Poisson bracket, and $\R$ is included as the (central) ideal
of constant functions. For any (local) action of a Lie algebra $\tor$ on $M$
by hamiltonian vector fields, i.e., any Lie algebra morphism
$\tor\to\ham(M,\omega)$, this exact sequence pulls back to a central extension
\begin{equation*}
0\to \R\stackrel{\afs}{\to}\torh\to\tor\to 0
\end{equation*}
of Lie algebras---if $\tor$ is a Lie subalgebra of $\ham(M,\omega)$, then
$\torh\leq C^\infty(M,\R)$ is the space of hamiltonian generators $f$ of the
action (i.e., with $\grad_\omega f\in \tor$). Such an extension has an
interpretation in affine geometry: dual to the inclusion $\afs\colon\R \to
\torh$, we have a surjection $\afs\transp\colon \torh^*\to\R$ and hence an
exact sequence
\begin{equation*}
0\to \tor^*\to \torh^*\stackrel{\afs\transp}{\to} \R\to 0.
\end{equation*}
The inverse image of $1\in\R$ under $\afs\transp$ is an affine subspace $\As:=
(\afs\transp)^{-1}(1)$ of $\torh^*$, modelled on $\tor^*$.  The space of
affine functions $f\colon\As\to\R$ is canonically isomorphic to $\torh$: the
projection of $f$ to $\tor$, viewed as a linear form on $\tor^*$, is its
derivative $\d f$ (at every point of $\As$), and the constant functions are
$\afs(c), c\in \R$.

\begin{defn} Let $\tor\into\ham(M,\omega)$ be a (local, effective)
action by hamiltonian vector fields, and identify the extension $\torh$ of
$\tor$ with its image in $C^\infty(M,\R)$. The \textit{natural momentum map}
$\mu\colon M\to\As\sub\torh^*$ is defined by $\ip{\mu(x),f}=f(x)$ for $x\in M$
and $f\in \torh$.
\end{defn}
A (local, effective) action by hamiltonian vector fields is called a
\emph{hamiltonian} action if the extension $\torh\to\tor$ has a (Lie algebra)
splitting.  However, if the splitting is not unique, it can be more convenient
to work with the natural momentum map rather than its projection onto $\tor^*$
using a splitting. If $\tor$ is abelian, then the action is hamiltonian iff
$\torh$ is also abelian, i.e., the hamiltonian generators Poisson commute.

\end{document}